\documentclass[11pt,draft]{amsart}

\usepackage{amsmath, amsfonts, amssymb, amsthm}

\newtheorem{theorem}{Theorem}[section]
\newtheorem{proposition}[theorem]{Proposition}
\newtheorem{corollary}[theorem]{Corollary}
\newtheorem{lemma}[theorem]{Lemma}

\theoremstyle{remark}
\newtheorem{notation}[theorem]{Notation}
\newtheorem{construction}[theorem]{Construction}
\newtheorem{definition}[theorem]{Definition}
\newtheorem{remark}[theorem]{Remark}
\newtheorem{example}[theorem]{Example}

\newcommand{\cB}{ \mathcal{B} }
\newcommand{\bB}{ \mathbb{B} }
\newcommand{\bC}{ \mathbb{C} }

\newcommand{\bE}{ \mathbb{E} }
\newcommand{\cF}{ \mathcal{F} }
\newcommand{\cH}{ \mathcal{H} }

\newcommand{\cM}{ \mathcal{M} }

\newcommand{\bR}{ \mathbb{R} }

\newcommand{\cX}{ \mathcal{X} }

\newcommand{\bZ}{ \mathbb{Z} }

\newcommand{\fV}{ \mathfrak{V} }

\newcommand{\cCP}{\mathcal{CP}}

\newcommand{\term}{ \mbox{term} }

\newcommand{\Int}{ \mbox{Int} }
\newcommand{\oo}{ \mbox{o} }

\newcommand{\Dist}{ \Sigma_{alg} }
\newcommand{\Moeb}{ \mbox{Moeb} }
\newcommand{\Reta}{ \widehat{\mbox{RB}} }
\newcommand{\Serie}{ \mbox{Ser} }
\newcommand{\RtoM}{ \widehat{\mbox{RM}} }
\newcommand{\EtatoM}{ \widehat{\mbox{BM}} }

\newcommand{\ip}[2]{\left \langle{#1},{#2} \right \rangle}
\newcommand{\set}[1]{\left\{#1\right\}}
\newcommand{\abs}[1]{\left|{#1}\right|}
\newcommand{\norm}[1]{\left\|{#1}\right\|}

\newcommand{\iB}{\mathrm{1}}

\begin{document}

\title{Convolution powers in the  operator-valued framework}
\author{Michael Anshelevich}
\thanks{M.A. was supported in part by NSF grant DMS-0900935.
S.T.B. was supported in part by a Discovery Grant from NSERC,
Canada, and by a University of Saskatchewan start-up grant.
M.F. was supported in part by grant ANR-08-BLAN-0311-03 from
Agence Nationale de la Recherche, France.
A.N. was supported in part by a Discovery Grant from NSERC, 
Canada.}
\address{M. Anshelevich: 
Department of Mathematics, Texas A\&M University, 
College Station, TX 77843-3368, U.S.A.}
\email{manshel@math.tamu.edu}

\author{Serban T. Belinschi}
\address{S.T. Belinschi: Department of Mathematics and 
Statistics, University of Saskatchewan, 106 Wiggins Road, Saskatoon,
Saskatchewan, S7N 5E6, Canada, and
\newline
Institute of Mathematics ``Simion Stoilow'' of the Romanian 
Academy.}
\email{belinsch@math.usask.ca}

\author{Maxime Fevrier}
\address{M. Fevrier: Institut de Math\'ematiques de Toulouse, 
Equipe de Statistique et Probabilit\'es, F-31062 Toulouse
Cedex 09, France.}
\email{fevrier@math.univ-toulouse.fr}

\author{Alexandru Nica}
\address{A. Nica: 
Department of Pure Mathematics, University of Waterloo, 
Waterloo, Ontario, N2L 3G1, Canada.}
\email{anica@math.uwaterloo.ca}

\subjclass[2000]{Primary 46L54.}  % Secondary}

\date{July 2011}

\begin{abstract}
We consider the framework of an operator-valued noncommutative
probability space over a unital $C^*$-algebra $\cB$. We show how for a
$\cB$-valued distribution $\mu$ one can define convolution powers
$\mu^{\boxplus \eta}$ (with respect to free additive convolution)
and $\mu^{\uplus \eta}$ (with respect to Boolean convolution),
where the exponent $\eta$ is a suitably chosen linear map from $\cB$
to $\cB$, instead of being a non-negative real number. More precisely,
$\mu^{\uplus \eta}$ is always defined when $\eta$ is completely
positive, while $\mu^{\boxplus \eta}$ is always defined when
$\eta - 1$ is completely positive (with ``$1$'' denoting the
identity map on $\cB$).

In connection to these convolution powers we define an evolution
semigroup $\{ \bB_{\eta} \mid \eta : \cB \to \cB$, completely
positive$\}$, related to the Boolean Bercovici-Pata bijection. We
prove several properties of this semigroup, including its connection
to the $\cB$-valued free Brownian motion.

We also obtain two results on the operator-valued analytic function
theory related to convolution powers $\mu^{\boxplus \eta}$. One of
the results concerns the analytic subordination
of the Cauchy-Stieltjes transform of $\mu^{\boxplus \eta}$ with
respect to the Cauchy-Stieltjes transform of $\mu$. The other one
gives a $\cB$-valued version of the inviscid Burgers equation, which
is satisfied by the Cauchy-Stieltjes transform of a $\cB$-valued
free Brownian motion.
\end{abstract}

\maketitle

$\ $

\section{Introduction}

\subsection{Convolution powers with respect to
\boldmath{$\boxplus$} }

The study of free probability was initiated by Voiculescu in the
early 1980s, and combinatorial methods introduced to it by Speicher
in the early 1990s. Soon thereafter both the analytic and the
combinatorial aspects of the theory were extended (in \cite{V1995} and
respectively \cite{S1998}) to an operator-valued framework.
Very roughly, the operator-valued framework is analogous to
conditional probability: instead of working with an expectation
functional (for noncommutative random variables) which takes values
in $\bC$, one works with a conditional expectation taking values in
an algebra $\cB$. The $\cB$-valued framework adds further depth to the
theory; a notable example of this appears for instance in the
relation between free probability and random matrices, where the
paper of Shlyakhtenko \cite{Sh1996} found relations between
operator-valued free probability and random band matrices.

An important role in free probability is played
by the free additive convolution $\boxplus$. This is an operation
on distributions which reflects the operation of addition for
free random variables. When considered in connection to bounded
selfadjoint variables in $\bC$-valued framework, $\boxplus$ is an
operation on compactly supported probability measures on $\bR$, and
was studied from the very beginning of the theory \cite{V1985,V1986}.
Also from the very beginning, Voiculescu \cite{V1986} introduced the
concept of $R$-transform $R_{\mu}$ for such a probability measure
$\mu$, and proved the linearization property that
$R_{\mu \boxplus \nu} = R_{\mu} + R_{\nu}$.
By using the $R$-transform one can moreover introduce convolution
powers with respect to $\boxplus$: for a compactly supported
probability measure $\mu$ and a real number $t \geq 0$, the
convolution power $\mu^{\boxplus t}$ (when it exists) is the
probability measure determined uniquely by the fact that
\begin{equation}   \label{eqn:1.1}
R_{\mu^{\boxplus t}} = t \cdot R_{\mu}.
\end{equation}

It is known \cite{NS1996} that $\mu^{\boxplus t}$ is always defined
when $t \geq 1$. On the other hand there exist special distributions
(the so-called infinitely divisible ones) where the ``$\boxplus t$''
powers are defined for every $t \geq 0$. An important such example
is provided by the standard semicircular distribution $\gamma$,
which is the analog in free probability for the normal law. If we
denote by $\gamma_t$ the centered semicircular distribution with
variance $t$ (so that $\gamma$ becomes $\gamma_1$) then, in full
analogy to the heat semigroup, these form a semigroup:
$\gamma_t \boxplus \gamma_s = \gamma_{t+s}$ for every $s,t \geq 0$.
In other words, one has that
\begin{equation}  \label{eqn:1.2}
\gamma_t = \gamma^{\boxplus t}, \ \ \forall \,  t \geq 0.
\end{equation}

Now let us move to operator-valued framework. Throughout the paper,
$\cB$ will be a fixed unital $C^{*}$-algebra. The $\cB$-valued
semicircular distributions form a class of examples that are
relatively well understood (see e.g. \cite{Sh1999}).
Here the centered semicircular
distributions are again indexed by their variances, but now these
variances are allowed to be arbitrary completely positive maps
$\eta : \cB \rightarrow \cB$ (!) This motivates a question,
explicitly asked by Hari Bercovici, of defining convolution
powers $\mu^{\boxplus \eta}$ for general $\cB$-valued distributions
$\mu$, so that in particular one obtains the $\cB$-valued analogue
of Equation (\ref{eqn:1.2}):
\begin{equation}  \label{eqn:1.3}
\gamma_\eta = \gamma^{\boxplus \eta},
\end{equation}
where ``$\gamma$'' stands now for the centered semicircular
distribution having variance equal to the identity map on $\cB$.

The answer to this question is one of the main results of this paper,
Theorem~\ref{Thm:alpha-geq-1}. Since the rescaling by $t$ from
Equation (\ref{eqn:1.1}) can be viewed as a
particular case of composing the $R$-transform with a
positive map from $\cB$ to $\cB$, we consider the problem of
defining convolution powers $\mu^{\boxplus \eta}$ via the formula
\begin{equation}   \label{eqn:1.4}
R_{\mu^{\boxplus \eta}} = \eta \circ R_{\mu},
\end{equation}
where $\eta : \cB \to \cB$ is a completely positive map.
Theorem~\ref{Thm:alpha-geq-1} says that (\ref{eqn:1.4}) meaningfully
defines a distribution $\mu^{\boxplus \eta}$ whenever
$\eta : \cB \to \cB$ is such that $\eta - 1$ is a completely positive
map (and where ``$1$'' stands for the identity map on $\cB$).

Here is the moment to make a clarification of our notations. We
will use the notation $\Sigma(\cB)$ for the space of all
$\cB$-valued distributions; the elements of $\Sigma(\cB)$ are
thus positive $\cB$-bimodule maps
$\mu : \cB \langle \cX \rangle \to \cB$
(see Section 2 below for notational details). A smaller class
$\Sigma^0(\cB)$ corresponds to the compactly supported distributions
from the $\bC$-valued case. On the other hand we will use the notation
$\Sigma_{alg} ( \cB )$ for the larger space of all unital $\cB$-bimodule
maps $\mu : \cB \langle \cX \rangle \to \cB$. It is easily seen that
Equation (\ref{eqn:1.4}) can be invoked to define the convolution
power $\mu^{ \boxplus \eta} \in \Sigma_{alg} ( \cB )$ for every
$\mu \in \Sigma_{alg} ( \cB )$ and every linear map
$\eta : \cB \to \cB$. The point of Theorem~\ref{Thm:alpha-geq-1} is
that upon starting with $\mu$ in the smaller space $\Sigma(\cB)$ and
with $\eta$ as described above, the resulting convolution power
$\mu^{ \boxplus \eta}$ still belongs to $\Sigma(\cB)$. In order to
arrive to this point, a key role in our considerations will be played
by the interaction between free probability and a simpler form of
noncommutative probability, called Boolean probability. We
elaborate on this interaction in the next subsection.

\subsection{Relations with Boolean probability}

Boolean probability has an operation of Boolean convolution $\uplus$,
which reflects the operation of addition for Boolean independent
random variables. This in turn has a linearizing transform, which
will be denoted in this paper by $B_{\mu}$. (Usually the linearizing
transform for $\uplus$ is rather denoted as ``$\eta_{\mu}$'', but in
this paper $\eta$ is reserved for denoting variance linear maps on
$\cB$.) If $\mu$ is a compactly supported probability measure on
$\bR$ then one defines convolution powers with respect to $\uplus$
by using the suitable rescaling of the $B$-transform,
\begin{equation}   \label{eqn:1.5}
B_{\mu^{\uplus t}} = t \cdot B_{\mu}.
\end{equation}
So far this parallels very closely the development of
$\boxplus$-powers, with $B_{\mu}$ being the Boolean counterpart of
$R_{\mu}$. However, due to the simpler nature of Boolean probability,
one now has \cite{SW1997} that $\mu^{\uplus t}$ is defined for every
compactly supported probability measure on $\bR$ and every
$t \geq 0$. (In the Boolean world, every $\mu$ is $\uplus$-infinitely
divisible.) This fortunate fact can then be put to use in free
probability due to the existence of a special bijection, called the
Boolean Bercovici-Pata bijection \cite{BP1999}, which links the
$R$-transform to the $B$-transform. Moreover, the Boolean
Bercovici-Pata bijection can be incorporated \cite{BN2008b} into
a semigroup $\{ \bB_t \mid t \geq 0 \}$ of transformations on
probability measures, which are explicitly defined in terms of
convolution powers:
\begin{equation}  \label{eqn:1.6}
\bB_t ( \mu ) := \Bigl( \, \mu^{\boxplus (1+t)} \,
\Bigr)^{ \uplus (1+t)^{-1} },
\end{equation}
holding for every probability measure $\mu$ on $\bR$ and every
$t \geq 0$. The original Boolean Bercovici-Pata bijection is
$\bB = \bB_1$. The transformations $\bB_t$ are sometimes said to
give an ``evolution towards $\boxplus$-infinite divisibility'', due
to the fact that $\bB_t ( \mu )$ is $\boxplus$-infinite divisible
for every $\mu$ and whenever $t \geq 1$.

The considerations from the preceding paragraph were given in the
$\bC$-valued framework. But the interactions between free and Boolean
probability turn out to continue to hold when one goes to
operator-valued framework. Let us first concentrate on the sheer
algebraic and combinatorial aspects of how this happens.
Paralleling the $\boxplus$ case, it is easily seen that one can
define the convolution power
$\mu^{ \uplus \eta } \in \Sigma_{alg} ( \cB )$ for every
$\mu \in \Sigma_{alg} ( \cB )$ and every linear map
$\eta : \cB \to \cB$, by simply making the requirement that
\begin{equation}   \label{eqn:1.7}
B_{ \mu^{ \uplus \eta } } = \eta \circ B_{\mu}.
\end{equation}
It is however non-trivial how to combine the generalized convolution
powers from (\ref{eqn:1.4}) and (\ref{eqn:1.7}) in order to
create, for a general linear map $\alpha : \cB \to \cB$, a
transformation $\bB_{\alpha}$ on $\Dist ( \cB )$ which is the analogue
of $\bB_t$ from Equation (\ref{eqn:1.6}). This requires some departure
from the techniques used previously in the $\bC$-valued case, and is
achieved in Section 6 of the paper. In Theorem \ref{thm:6.4} we
prove that the transformations
$\bB_{\alpha} : \Dist ( \cB ) \to \Dist ( \cB )$ which are
obtained form a commutative (!) semigroup:
\begin{equation}   \label{eqn:1.8}
\bB_{\alpha} \circ \bB_{\beta} = \bB_{\alpha + \beta}, \ \
\forall \, \alpha , \beta : \cB \to \cB
\end{equation}
(we emphasize that in (\ref{eqn:1.8}) the linear maps
$\alpha$ and $\beta$ are not required to commute). We show moreover
that one has the formula
\begin{equation}   \label{eqn:1.9}
\Bigl( \, \bB_{\alpha} ( \mu ) \, \Bigr)^{\uplus (1+ \alpha)} =
\mu^{\boxplus (1+ \alpha)},
\end{equation}
holding for every $\mu \in \Dist ( \cB )$ and every linear map
$\alpha : \cB \to \cB$; so in the special case when
$1 + \alpha$ is invertible, one can raise both sides of
(\ref{eqn:1.9}) to the power $\uplus (1+ \alpha)^{-1}$ in
order to obtain a faithful analog of the formula (\ref{eqn:1.6})
from the $\bC$-valued framework.

While the space $\Dist ( \cB )$ provides a nice larger environment
which is good for algebraic manipulations, our interest really lies
in the smaller spaces $\Sigma ( \cB )$ and $\Sigma^{0} ( \cB )$,
consisting of distributions which can appear in $C^*$-framework.
Thus it is of certain interest to look at the case when the
linear maps $\alpha , \beta , \eta$ that were
considered above are completely positive (on the unital $C^*$-algebra
$\cB$ that is fixed throughout the paper), and to establish
conditions under which the corresponding convolution powers and/or
transformations $\bB_{\alpha}$ leave
$\Sigma ( \cB )$ and $\Sigma^{0} ( \cB )$ invariant. We obtain this
by using some suitable {\em operator models} for distributions in
$\Sigma ( \cB )$. Two such (relatively simple) operator models are
described in Section 7 of the paper, and are used to prove that:

(i) $\mu^{ \uplus \eta } \in \Sigma(\cB)$ whenever
$\mu \in \Sigma(\cB)$ and $\eta : \cB \to \cB$ is a completely
positive linear map (cf. Theorem \ref{thm:7.5});

(ii) $\bB_{\alpha} ( \mu ) \in \Sigma ( \cB )$ whenever
$\mu \in \Sigma(\cB)$ and $\alpha : \cB \to \cB$ is a completely
positive linear map (cf. Theorem \ref{thm:7.8}).

Finally, we can now return to the point left at the end of
subsection 1.1, and explain why is it that
$\mu^{\boxplus \eta} \in \Sigma ( \cB )$ whenever
$\mu \in \Sigma ( \cB )$ and $\eta : \cB \to \cB$ is such that
$\eta -1$ is completely positive: denoting $\eta - 1 =: \alpha$,
we see from (\ref{eqn:1.9}) that the distribution
$\mu^{\boxplus \eta}$ (which a priori lives in the larger space
$\Dist ( \cB )$) satisfies
\[
\mu^{\boxplus \eta} =
\Bigl( \, \bB_{\alpha} ( \mu ) \, \Bigr)^{\uplus \eta};
\]
but the distribution on the right-hand side of the latter equation
does belong to $\Sigma ( \cB )$, by virtue of the results (i) and
(ii) indicated above.

$\ $

\subsection{Relations to free Brownian motion, and to
analytic functions}

Consider again the $\bC$-valued framework, and the semigroup of
semicircular distributions $\{ \gamma_t \mid t \geq 0 \}$ from
Equation (\ref{eqn:1.2}). For a probability measure $\mu$ on $\bR$,
the process $\{ \mu \boxplus \gamma_t \mid t \geq 0 \}$ is called
the {\em free Brownian motion started at $\mu$}. In
\cite{BN2008b, BN2009} it was observed that the transformations
$\bB_t$ from (\ref{eqn:1.6}) are related to the free Brownian
motion via an evolution equation of the form
\[
\bB_t \bigl( \, \Phi ( \mu ) \, \bigr) =
\Phi ( \, \mu \boxplus \gamma_t \, ), \ \ t \geq 0,
\]
where $\Phi$ is a special transformation on probability measures
(not depending on $\mu$ or $t$). In this paper we introduce the
$\cB$-valued analog of the transformation $\Phi$ (defined on lines
similar to those of \cite{BN2009}), and we obtain the corresponding
evolution equation, where the role of time parameter is now taken
by a linear map $\eta : \cB \to \cB$. That is, we have that
\[
\bB_\eta [\Phi[\mu]] = \Phi[\mu \boxplus \gamma_\eta] ,
\]
holding for $\mu \in \Sigma ( \cB )$ and $\eta : \cB \to \cB$
completely positive. This is obtained in Theorem \ref{thm:7.11}
(in a plain algebraic version) and in Corollary \ref{cor:7.11}
(in the completely positive version).

$\ $

In a related development, Section 8 of the paper establishes some
results concerning the operator-valued analytic function theory
related to the convolution powers $\mu^{\boxplus \eta}$.

On the one hand we show that for $\eta$ as in
Theorem~\ref{Thm:alpha-geq-1}
(that is, an $\eta$ such that $\eta -1$ is completely positive)
one has analytic subordination of the Cauchy-Stieltjes transform
of $\mu^{\boxplus \eta}$ with respect to the Cauchy-Stieltjes
transform of $\mu$. This is the $\cB$-valued analogue of a known fact
from the $\bC$-valued framework \cite{BB2004, BB2005},
but where now the subordination function is an analytic self-map
of the set $\{ b \in \mathcal B \mid \Im b>0 \}$.

On the other hand, we show that the Cauchy-Stieltjes transform
of the $\cB$-valued free Brownian motion started at a
distribution $\mu \in \Sigma^{(0)} ( \cB )$ satisfies a $\cB$-valued
version of the inviscid Burgers equation. The occurrence of the
Burgers equation in free probability came with a fundamental
result of Voiculescu, where (in $\bC$-valued framework) the complex
Burgers equation was found to be the free analogue of the heat
equation. This means, more precisely, that the complex Burgers
equation with initial data the Cauchy-Stieltjes transform of a
given probability measure $\mu$ on $\bR$ is solved by the
Cauchy-Stieltjes transform of the free Brownian motion started at
$\mu$. Theorem \ref{thm:8.3} of the present paper establishes the
$\cB$-valued analog of this (in fact of a slightly stronger
result, given in \cite{BN2008b}, which is expressed in terms of the
transformations $\bB_t$).

\subsection{Organization of the paper}

Besides the present introduction, the paper has 8 other sections.
Section~2 reviews some general background from $\cB$-valued
noncommutative probability, then in Section~3 we set up the
$\cB$-series machinery and the nesting structures corresponding
to non-crossing partitions. Sections 4 and 5 set up basic definitions
and results concerning the operator-valued $R$ and $B$-transforms,
and concerning a transformation $\Reta$ (defined on $\cB$-series)
which connects them. The definition of convolution powers
$\mu^{\boxplus \eta}$ and $\mu^{\uplus \eta}$ is also given here.
In Section 6, the transformations $\bB_\eta$ are defined and shown
to form a semigroup, and the evolution equation is proved.
In Section 7 we describe the operator models and use them to
prove the positivity results mentioned in subsection 1.2 above.
Section 8 contains the results
on operator-valued analytic function theory which were announced in
subsection 1.3. Finally, the main result in Section 9 is an
alternative operator model for Boolean convolution powers.

$\ $

$\ $

\section{Preliminaries}

\setcounter{equation}{0}
\setcounter{theorem}{0}

\subsection{Free $\cB$-bimodule}
Throughout the paper $\cB$ will be a $C^\ast$-algebra. For $\cX$ a
formal variable, denote
\begin{equation}   \label{eqn:2.11}
\cB \langle \cX \rangle =
\cB \oplus \cB \cX \cB \oplus \cB \cX \cB \cX \cB \oplus \ldots
\end{equation}
the algebra of all polynomials in $\cX$, with coefficients in $\cB$,
and
\[
\cB_0 \langle \cX \rangle =
\cB \cX \cB \oplus \cB \cX \cB \cX \cB \oplus \ldots
\]
the polynomials without constant term. We will assume that $\cX$ and
$\cB$ are algebraically independent. This means (by definition)
that Equation (\ref{eqn:2.11}) amounts to
\[
\cB \langle \cX \rangle \simeq \bigoplus_{n=1}^\infty
\cB^{\otimes_{\bC}^n},
\]
the tensor product being over $\bC$. The set $\cB \langle \cX \rangle$ is a
$\cB$-bimodule in the obvious way.

In particular, a collection of $\bC$-linear maps
$\mu_n : \cB^{\otimes_{\bC} (n-1)} \rightarrow \cB$ can be combined via
\[
\mu[\cX b_1 \cX b_2 \cdots b_{n-1} \cX] =
\mu_n(b_1 \otimes b_2 \otimes \ldots \otimes b_{n-1})
\]
into a $\cB$-bimodule map
\[
\mu : \cB \langle \cX \rangle \rightarrow \cB
\]
such that $\mu[b] = b$ for $b \in \cB$. As already mentioned in the
introduction, the set of all such maps will be denoted by $\Dist(\cB)$.

If such a $\mu$ is positive, we will refer to it as a
\emph{conditional expectation}, and omit the term ``positive''.
Since $\cB$ is a $C^\ast$-algebra, by Proposition 3.5.4 of
\cite{S1998}, in this case $\mu$ is automatically completely positive.
We will denote
\[
\Sigma(\cB) = \set{\text{(positive) conditional expectations }
\mu : \cB \langle \cX \rangle \rightarrow \cB}.
\]

\subsection{$\cB$-series}

\begin{definition}  \label{def:2.31}

$1^o$ We will use the name {\em $\cB$-series} for objects of the form
\begin{equation}   \label{eqn:2.311}
F = ( \beta_n )_{n \geq 1},
\end{equation}
where, for every $n \geq 1$, $\beta_n : \cB^{n-1} \to \cB$ is a
$\bC$-multilinear functional (with the convention that $\beta_1$
is an element of $\cB$). The set of all $\cB$-series will be
denoted by $\Serie ( \cB )$.

$2^o$ For a series $F$ as in (\ref{eqn:2.311}), the functionals
$\beta_n$ will be referred to as {\em terms} of $F$. We will use
the notation
\begin{equation}   \label{eqn:2.312}
F^{[n]} := \bigl( \, \mbox{$n$-th term of $F$} \, \bigr), \ \
\mbox{ for $F \in \Serie ( \cB )$ and $n \geq 1$.}
\end{equation}

$3^o$ On $\Serie ( \cB )$ we have some natural operations. In
the present paper it is particularly important that for every
$\bC$-linear map $\alpha : \cB \to \cB$ and every
series $F \in \Serie ( \cB )$ one can define the composition
$\alpha \circ F \in \Serie ( \cB )$ by putting
\begin{equation}   \label{eqn:2.313}
( \alpha \circ F )^{ [n] } :=  \alpha \circ F^{ [n] }, \ \
\forall \, n \geq 1
\end{equation}
(for $n=1$, this just means that
$( \alpha \circ F )^{ [1] } :=  \alpha ( \, F^{[1]} \, ) \in \cB$).

It is also clear that for $F_1, F_2 \in \Serie ( \cB )$ and
$\lambda_1 , \lambda_2 \in \bC$ one can define the linear combination
$\lambda_1 F_1 + \lambda_2 F_2 \in \Serie ( \cB )$ by putting
\begin{equation}   \label{eqn:2.314}
( \lambda_1 F_1 + \lambda_2 F_2 )^{ [n] } :=
\lambda_1 \cdot F_1^{[n]} + \lambda_2 \cdot F_2^{ [n] } , \ \
\forall \, n \geq 1.
\end{equation}
This extends naturally to a $\cB$-bimodule structure, where for
$b \in \cB$ and $F \in \Serie ( \cB )$ the new series $bF$ and
$Fb$ are obtained by taking the linear map $\alpha$ from
(\ref{eqn:2.313}) to be given by left (respectively right)
multiplication with $b$ on $\cB$.
\end{definition}

\begin{definition}   \label{def:2.32}
Let $\mu$ be a distribution in $\Dist ( \cB )$. The
{\em moment series} of $\mu$ is the series
$M_{\mu} \in \Serie ( \cB )$ with terms defined as follows:
$M_{\mu}^{ [1] } = \mu ( \cX )$ and
\begin{equation}   \label{eqn:2.321}
M_{\mu}^{ [n] } (b_1, \ldots , b_{n-1}) =
\mu [ \cX b_1 \cX b_2 \cdots \cX b_{n-1} \cX ]
\end{equation}
for every $n \geq 2$ and $b_1, \ldots , b_{n-1} \in \cB$.
\end{definition}

\begin{remark}   \label{rem:2.33}
Clearly, the correspondence $\mu \mapsto M_{\mu}$ is a bijection
between $\Dist(\cB)$ and $\Serie(\cB)$. The example of moment
series also explains why in Definition \ref{def:2.31} we used the
notation $F^{[n]}$ for a function of $n-1$ arguments (the functional
$M_{\mu}^{ [n] } : \cB^{n-1} \to \cB$ really is some kind of
``moment of order $n$'' for $\mu$).
\end{remark}

\begin{remark}   \label{rem:2.34}
Let $F$ be a $\cB$-series and let $b$ be an element of $\cB$.
In preparation of analytic considerations that will show up later
in the paper, we mention that we will use the notation
\[
F(b) := F^{ [1] } + \sum_{n=2}^{\infty} F^{[n]}(b, \dots,b) \in \cB
\]
whenever the sum on the right-hand side of this equality converges
(in the norm topology of the $C^{*}$-algebra $\cB$).
\end{remark}

\begin{remark}   \label{rem:2.35}
In order to justify the terminology introduced above, let us look
for a moment at what this amounts to in the special case when
$\cB = \bC$. In this case what one does
is to take a series with complex coefficients
\begin{equation}   \label{eqn:2.351}
f(z) = \sum_{n=1}^{\infty} \ \alpha_{n} z^{n-1},
\end{equation}
and convert it into $F = ( F^{[n]} )_{n \geq 1}$,
with $F^{[1]} = \alpha_1 \in \bC$ and where
$F^{[n]} : \bC^{n-1} \to \bC$ is defined by
\[
F^{[n]} (z_1, \ldots, z_{n-1}) = \alpha_n z_1 \cdots z_{n-1},
\ \ \forall \, n \geq 2 \mbox{ and }
z_1, \ldots , z_{n-1} \in \bC .
\]
So here every $F^{[n]}$ is indeed a ``term'' in the writing of
the series $f(z)$ (under suitable convergence hypotheses, the
right-hand side of (\ref{eqn:2.351}) is the infinite sum
$\sum_{n=1}^{\infty}  F^{[n]} (z,z, \ldots , z)$).
\end{remark}

\subsection{Distributions in a $\cB$-valued
$C^{*}$-probability space}

A {\em $\cB$-valued $C^\ast$-probability space} is a
pair $(\cM, \bE)$ where
$\cM$ is a $C^\ast$-algebra such that $\cM \supseteq \cB$ and
$\bE: \cM \rightarrow \cB$ is a conditional expectation (a unital
positive $\cB$-bimodule map).

If $(\cM, \bE)$ is a $\cB$-valued $C^\ast$-probability space and
if $X = X^{*} \in \cM$, then one defines the {\em distribution} of
$X$ as the conditional expectation
$\mu_X \in \Sigma(\cB)$ given by
\[
\mu_X[b_0 \cX b_1 \cX \cdots \cX b_n]
= \bE[b_0 X b_1 X \cdots X b_n].
\]
We will denote by
\[
\Sigma^0(\cB) \subset \Sigma(\cB)
\]
the set of all $\mu$ arising in this way.
Equivalently, $\mu \in \Sigma^0(\cB)$
if for any state $\phi$ on $\cB$, the operator $\cX$ in the GNS
representation of $(\cB\langle \cX \rangle, \phi \circ \mu)$ is
bounded. More explicitly this is the case if for some $M > 0$ and
all $b_1, b_2, \ldots, b_{n-1} \in \cB$,
\begin{equation}
\label{Eq:Bounded}
\norm{\mu[\cX b_1 \cX \ldots b_{n-1} \cX]} \leq M^n \norm{b_1} \cdot \norm{b_2} \cdot \ldots \cdot \norm{b_{n-1}}.
\end{equation}

Moreover, the {\em moment series} of $X$ is defined to be
\[
M_{ { }_X } := M_{ \mu_{X} } \in \Serie ( \cB ).
\]
Or in other words, one has $M_{ { }_X }^{[1]} = \bE [X]$ and
\[
M_{ { }_X }^{[n]} (b_1, \ldots , b_{n-1}) =
\bE[X b_1 X b_2 \cdots X b_{n-1} X]
\]
for every $n \geq 2$ and $b_1, \ldots , b_{n-1} \in \cB$.
It is worth noting that
\begin{equation}\label{M-series}
M_{ { }_X } (b)=\sum_{n=1}^\infty M_{ { }_X}^{[n]}(b,\dots,b)
\end{equation}
is an analytic map on $\{b\in\mathcal B\colon\|b\|<\|X\|^{-1}\}$.

We also mention here that the {\em generalized resolvent} (or
{\em operator-valued Cauchy-Stieltjes transform}) of the
distribution $\mu_X$ of a random variable $X$ is defined by
\begin{equation}\label{eq:2.10}
G_{\mu_X}(b)=\mathbb E\left[(b-X)^{-1}\right].
\end{equation}
This function of $b$ is analytic on the set of elements $b\in\cB$
for which $b-X$ is invertible. The Cauchy-Stieltjes transform
is particularly relevant in the context of fully matricial sets and
maps \cite{V2000}. Its natural domain in the $C^*$-algebraic context
is the set $\{b\in\mathcal B\colon\Im b>0\}$. (By $\Im b>0$ we mean
that there exists some $\varepsilon>0$ so that $\Im b=(b-b^*)/2i\ge
\varepsilon\cdot1$.) However, the equality
$$
G_{\mu_X}(b)=b^{-1}(M_X(b^{-1})+1)b^{-1}
$$
is easily seen to be true for $b$ invertible with $\|b^{-1}\|<\|X\|^{-1}.$

\subsection{Independence and convolution}

Let $(\cM, \bE)$ be a $\cB$-valued $C^\ast$-probability space,
and let $\cM_1, \cM_2, \ldots, \cM_k$ be subalgebras of $\cM$ which
contain $\cB$. These subalgebras are said
to be {\em freely independent} with respect to $\bE$ when the
following happens: for any $j(i) \neq j(i+1)$, $X_i \in \cM_{j(i)}$
with $\bE[X_i] = 0$, and for any $b_0, \ldots b_n \in \cB$, we have
\[
\bE[b_0 X_1 b_2 X_2 \ldots X_n b_n] = 0.
\]
Operators are freely independent if the $\ast$-subalgebras they
generate over $\cB$ are freely independent.

Let $(\cM, \bE)$ be a $\cB$-valued $C^\ast$-probability space, and
suppose we have a decomposition $\cM = \cB \oplus \cM_0$ (with
multiplication $(b_1, m_1) \cdot (b_2, m_2) =
(b_1 b_2, b_1 m_2 + m_1 b_2 + m_1 m_2)$). Subalgebras
$\cM_1, \cM_2, \ldots, \cM_k \subset \cM_0$ are
{\em Boolean independent} with respect to $\bE$ if for any
$j(i) \neq j(i+1)$, $X_i \in \cM_{j(i)}$ for $i \neq 1, n$,
$X_i \in \cB \oplus \cM_{j(i)}$, $i = 1, n$, and any
$b_0, \ldots b_n \in \cB$ we have
\[
\bE[b_0 X_1 b_2 X_2 \ldots X_n b_n] =
b_0 \bE[X_1] b_1 \bE[X_2] \ldots \bE[X_n] b_n.
\]
Operators in $\cM_0$ are Boolean independent if the $\ast$-subalgebras
of $\cM_0$ they generate over $\cB$ are Boolean independent.

If $\mu, \nu \in \Sigma(\cB)$, there exist freely independent symmetric (possibly unbounded) operators $X, Y$ with $\mu_X = \mu$, $\mu_Y = \nu$. The distribution of $X + Y$ is uniquely determined by $\mu$ and $\nu$, and is their free convolution:
\[
\mu \boxplus \nu := \mu_{X+Y} \in \Sigma(\cB).
\]
Similarly, if $X$ and $Y$ are chosen Boolean independent, their distribution is the Boolean convolution of $\mu$ and $\nu$,
\[
\mu \uplus \nu := \mu_{X+Y} \in \Sigma(\cB).
\]

As mentioned in the introduction, both $\boxplus$ and $\uplus$ have
linearizing transforms, the $R$-transform and respectively the
$B$-transform. For $\mu \in \Sigma(\cB)$ (and more generally, for
$\mu \in \Sigma_{alg} ( \cB )$) the $R$-transform $R_{\mu}$ and the
$B$-transform $B_{\mu}$ are series in $\Serie ( \cB )$; their precise
definitions will be reviewed in Section~\ref{Section:Transforms}
below. The linearization property is that
\[
R_{\mu \boxplus \nu} = R_\mu + R_\nu
\mbox{ and respectively }
B_{\mu \uplus \nu} = B_\mu + B_\nu,
\mbox{ for every } \mu , \nu \in \Sigma(\cB).
\]

\section{$\cB$-series and non-crossing partitions}

\setcounter{equation}{0}
\setcounter{theorem}{0}

\begin{remark}   \label{rem:3.3}
($NC(n)$ terminology.)

The workhorse for combinatorial considerations in free probability
is the set $NC(n)$ of non-crossing partitions of
$\{ 1, \ldots , n \}$. In connection to it we will use the
the standard notations and terminology, as appearing for instance
in Lecture 9 of the monograph \cite{NS2006}. In particular the
partitions in $NC(n)$ will be denoted by letters like
$\pi, \rho , \ldots$ (typical notation will be
$\pi = \{ V_1, \ldots , V_k \} \in NC(n)$, where the $V_i$ are
the blocks of $\pi$). We will also use the customary partial order
given on $NC(n)$ by reverse refinement: for $\pi , \rho \in NC(n)$
we write ``$\pi \leq \rho$'' to mean that every block of $\rho$ is
a union of blocks of $\pi$. The minimal and maximal element of
$( NC(n), \leq )$ are denoted by $0_n$ (the partition of
$\{ 1, \ldots , n \}$ into $n$ singleton blocks) and respectively
$1_n$ (the partition of $\{ 1, \ldots , n \}$ into only one block).
\end{remark}

\begin{remark}   \label{rem:3.4}
(Nested terms $F^{[ \pi ]}$ for $F \in \Serie ( \cB )$ and
$\pi \in NC(n)$.)

Let a series $F \in \Serie ( \cB )$ be given. In this remark and in
the next definition we explain how one naturally constructs a family
of $\bC$-multilinear functionals $F^{ [ \pi ] } : \cB^{n-1} \to \cB$,
one such functional for every $n \geq 1$ and every $\pi \in NC(n)$.
If $\pi$ happens to be $1_n$ then we will just get
$F^{ [1_n] } = F^{ [n] }$, the $n$-th term of the series $F$. For a
general $\pi \in NC(n)$,
the point of view that works best when defining $F^{[ \pi ]}$ is to
treat $\pi$ as a ``recipe for nesting intervals inside each other''.
Indeed, the idea of nesting intervals has a correspondent in the
framework of multilinear functionals, where such functionals are
nested inside each other by using parentheses. (Thus if $\pi$ is
written explicitly, $\pi = \{ V_1, \ldots , V_k \}$, then
$F^{[ \pi ]}$ will be obtained by suitably nesting inside each
other the functionals
$F^{[ \ |V_1| \ ]}, \ldots , F^{[ \ |V_k| \ ]}$.)
This very fundamental relation between non-crossing partitions
and multilinear functionals arising in $\cB$-valued noncommutative
probability was put into evidence in \cite{S1998}.

We first explain how the things go on a concrete example. Consider
a functional like $L: \cB^4 \to \cB$, where
\[
L( b_1, \ldots , b_4 ) :=
F^{[3]} ( \, b_1 F^{[2]} (b_2) b_3 \, , \, b_4 \, ), \ \ \mbox{ for }
b_1, \ldots , b_4 \in \cB .
\]
In a very ``literal'' sense (from the point of view of a typesetter)
 the right-hand side of the above formula is of the form
\[
W_1 b_1 W_2 b_2 W_3 b_3 W_4 b_4 W_5
\]
where each of $W_1, \ldots , W_5$ is a string of symbols made out
of left and right parentheses, commas, and the occasional
``$F^{[m]}$''. More precisely, we have
\[
\left\{  \begin{array}{l}
\mbox{$W_1$ = ``$F^{[3]} ( \,$'' }      \\
\mbox{$W_2$ = ``$F^{[2]} ( \,$'' }    \\
\mbox{$W_3$ = ``$\, ) \, $'' (only a right bracket) }      \\
\mbox{$W_4$ = ``$\, , \, $'' (only a comma) }              \\
\mbox{$W_5$ = ``$\, ) \, $'' (only a right bracket). }     \\
\end{array}   \right.
\]
Conversely, suppose that somebody was to give us the words
$W_1, \ldots , W_5$ listed above; then we could write
down mechanically the sequence
$W_1 b_1 W_2 b_2 W_3 b_3 W_4 b_4 W_5$, after which we could read
the result as a legit expression defining a functional from
$\cB^4$ to $\cB$.

Now, $L$ from the preceding paragraph turns out to be precisely
the functional $F^{[ \pi ]}$ which corresponds to our fixed series
$F$ and the non-crossing partition
$\pi = \{ \, \{ 1,4,5 \}, \, \{ 2,3 \} \, \} \in NC(5)$.
This is because the words $W_1, \ldots , W_5$ are exactly those
created by starting with this special $\pi$ and by applying the
rules described in the next definition.
\end{remark}

\begin{definition}    \label{def:3.5}
Let $F$ be a series in $\Serie ( \cB )$ and let $\pi$ be a
partition in $NC(n)$. For every $1 \leq m \leq n$ we define a
string of symbols, $W_m$, according to the following rules.

\begin{itemize}

\item
If $m$ is the minimum element of block $V$ of $\pi$ with
$|V| = k \geq 2$, then $W_m := ``F^{ [k] } ( $''.

\item
If $m$ is the maximum element of block $V$ of $\pi$ with
$|V| \geq 2$, then $W_m$ = ``$ \, ) \,$'' (just a right bracket).

\item
If $m$ belongs to a block $V$ of $\pi$ where
$\min (V) < m < \max (V)$, then $W_m$ = ``$ \, , \,$'' (just a comma).

\item
If $m$ forms by itself a singleton block of $\pi$, then
$W_m$ = ``$F^{ [1] }$'' (no parentheses or comma besides the
occurrence of $F^{ [1] }$).

\end{itemize}

\noindent
The $\bC$-multilinear functional $F^{ [ \pi ] } : \cB^{n-1} \to \cB$
is then defined as follows: given $b_1, \ldots , b_{n-1} \in \cB$
we form the string of symbols obtained by concatenating
\[
W_1 b_1 W_2 b_2 \cdots W_{n-1} b_{n-1} W_n;
\]
then we read this as a parenthesized expression which produces
an element $b \in \cB$, and we define
$F^{ [ \pi ] } (b_1, \ldots , b_{n-1})$ to be equal to this $b$.
\end{definition}

\begin{remark}    \label{rem:3.6}
The special case $\pi = 1_n$ of the above definition leads
to the formula
\[
F^{ [1_n] } =  F^{ [n] }, \ \ \forall n \geq 1.
\]
Indeed, in this case the string of symbols
$W_1 b_1 W_2 b_2 \cdots W_{n-1} b_{n-1} W_n$
has $W_1$ = ``$F^{ [n] } ($'', has $W_n$ = ``$)$'', and all of
$W_2, \ldots , W_{n-1}$ are commas.

On the other hand for $\pi = 0_n$ we get the formula
\[
F^{ [0_n] } (b_1, \ldots , b_{n-1} ) =
F^{ [1] } b_1 F^{ [1] } b_2 \cdots
F^{ [1] } b_{n-1} \cdots F^{ [1] }
\]
(holding for every $n \geq 2$ and $b_1, \ldots , b_{n-1} \in \cB$).

Let us show one more concrete example, illustrating how the
nestings and concatenations of blocks of $\pi$ generate a
parenthesized expression. Say that $n=6$ and that
$\pi = \{ \, \{ 1,3,4 \}, \{ 2 \}, \{ 5,6 \} \, \} \in NC(6)$.
Then the list of words $W_1, \ldots , W_6$ used to define
$F^{ [ \pi ] }$ goes like this:

\begin{center}
$W_1$ = ``$F^{ [3] } ($'',
$W_2$ = ``$F^{ [1] }$'',
$W_5$ = ``$F^{ [2] } ($'',
\end{center}

\noindent
while $W_4$ and $W_6$ are right parentheses and $W_3$ is a comma.
So the string of symbols $W_1 b_1 \cdots W_5 b_5 W_6$ gives us here
the formula
\[
F^{ [ \pi ] } (b_1, \ldots , b_5) =
F^{ [3] } ( b_1 F^{ [1] } b_2 , b_3 ) b_4
F^{ [2] } ( b_5 ), \ \ \forall \, b_1, \ldots , b_5 \in \cB .
\]

We next record a general fact which follows from
the procedure of constructing the functionals $F^{ [ \pi ] }$,
and which will be used repeatedly in the sequel.
\end{remark}

\begin{lemma}  \label{lemma:3.8}
Let $F$ be a series in $\Serie ( \cB )$ and let
$\rho = \{ V_1, \ldots , V_k \}$ be a partition in $NC(n)$.
Consider the formula defining the multilinear functional
$F^{ [ \rho ] } : \cB^{n-1} \to \cB$. This formula has embedded
in it some occurences of the functionals
$F^{ [ \, |V_1| \, ] }, \ldots , F^{ [ \, |V_k| \, ] }$.
Suppose that we take some partitions
$\pi_1 \in NC( |V_1| ), \ldots , \pi_k \in NC( |V_k| )$,
and that for every $1 \leq j \leq k$ we replace the functional
$F^{ [ \, |V_j| \, ] }$ by the functional $F^{ [ \pi_j ] }$,
inside the formula for $F^{ [ \rho ] }$. Then the formula
defining $F^{ [ \rho ] }$ is transformed into the formula for
$F^{ [ \pi ] }$, with $\pi \in NC(n)$ defined as follows:
\[
\pi = \widehat{\pi_1} \cup \cdots \cup \widehat{\pi_k},
\]
where, for every $1 \leq j \leq k$, we denote by
$\widehat{\pi_j} \in NC(V_j)$ the non-crossing partition
obtained by relabelling $\pi_j$.
\end{lemma}

\begin{proof}
Look at the string of symbols
$W_1 b_1 \cdots W_{n-1} b_{n-1} W_n$
which is used in the definition of $F_{\rho}$, and follow how
the words $W_1, \ldots , W_n$ are changed when one replaces every
$F^{ [ \, |V_j| \, ] }$ by $F^{ [ \pi_j ] }$, $1 \leq j \leq k$.
It is immediate that the ensuing string of words is exactly the
one which appears in the definition of $F^{ [ \pi ] }$.
\end{proof}

\begin{remark}   \label{rem:3.9}
At some points throughout the paper we will need a variation of the
construction of $F^{ [ \pi ] }$ which involves {\em coloured}
non-crossing partitions.  For the sake of simplicity, we discuss here
the situation of colourings which use two colours. Given a
partition $\pi \in NC(n)$, a colouring of $\pi$ is then a map
$c : \pi \to \{ 1,2 \}$ (that is, a procedure which associates to
every block $V$ of $\pi$ a number $c(V) \in \{ 1, 2 \}$). For such
$\pi$ and $c$ one can talk about ``mixed nested functionals'' of the
form
\begin{equation}  \label{eqn:3.91}
( F,G )^{ [ \pi, c ] }
\end{equation}
where $F,G$ are two series in $\Serie ( \cB )$. The object in
(\ref{eqn:3.91}) is a multilinear functional from
$\cB^{n-1}$ to $\cB$, constructed by the same method as in
Definition \ref{def:3.5}, but which uses some terms of $F$ and
some terms of $G$ (depending on what is the colour of the
corresponding block of $\pi$ -- blocks of colour 1 go with $F$,
and blocks of colour 2 go with $G$).

Concrete example: take again the case when
$\pi = \{ \, \{ 1,3,4 \}, \{ 2 \}, \{ 5,6 \} \, \} \in NC(6)$,
as we illustrated at the end of Remark \ref{rem:3.6}. Suppose that
$\pi$ is coloured so that $c( \, \{ 1,3,4 \} \, ) = 1$, while
$c( \, \{ 2 \} \, ) = c( \, \{ 5,6 \} \, ) = 2$. Then
the list of words $W_1, \ldots , W_6$ that we use is changed
in the respect that we now have

\begin{center}
$W_1$ = ``$F^{ [3] } ($'',
$W_2$ = ``$G^{ [1] }$'',
$W_5$ = ``$G^{ [2] } ($'',
\end{center}

\noindent
(while $W_4$ and $W_6$ still are right parentheses, and $W_3$ is a
comma). The string of symbols $W_1 b_1 \cdots W_5 b_5 W_6$ thus
gives the formula
\[
( F, G )^{ [ \pi, c ] } (b_1, \ldots , b_5) =
F^{ [3] } ( \, b_1 G^{ [1] } b_2 \, , \, b_3 \, ) \, b_4 \,
G^{ [2] } ( b_5 ), \ \ \forall \, b_1, \ldots , b_5 \in \cB .
\]

Clearly, the same idea of colouring can be used when more than two
colours are involved, in order to define for instance mixed linear
functionals of the form
\begin{equation}  \label{eqn:3.92}
( F,G,H )^{ [ \pi, c ] }
\end{equation}
where $F,G,H$ are series in $\Serie ( \cB )$, $\pi$ is in $NC(n)$,
and $c : \pi \to \{ 1,2,3 \}$ is a colouring of $\pi$ in three
colours.
\end{remark}

\begin{remark}   \label{rem:3.7}
(Interval partitions.)

At various points in the paper we will need to look at functionals
$F^{ [ \pi ] }$ (as introduced in Definition \ref{def:3.5}) in
the special, simpler, case when $\pi$ is an interval partition. We
record here the formula that is relevant for such a special case.

A partition $\pi = \{ V_1, \ldots , V_k \}$ of $\{ 1, \ldots , n \}$
is said to be an interval partition when every block $V_i$ is of
the form $[p,q] \cap \bZ$ for some $1 \leq p \leq q \leq n$. The
set of all interval partitions of $\{ 1, \ldots , n \}$ will be
denoted as $\Int (n)$. It is clear that $\Int (n) \subseteq NC(n)$,
but it is occasionally preferable to think of $\Int (n)$ as of a
partially ordered set in its own right, with partial order ``$\leq$''
still given by reverse refinement. It is easily seen that
$( \Int (n) , \leq )$ is then isomorphic to the partially ordered
set of subsets of $\{ 1, \ldots , n-1 \}$.

Now let $F$ be in $\Serie ( \cB )$ and let $\pi$ be a partition
in $\Int (n)$. Let us write explicitly
$\pi = \{ V_1, \ldots , V_k \}$, with the blocks $V_i$ picked
such that $\min (V_1) < \min (V_2) < \cdots < \min (V_k)$. Consider
the numbers $1 \leq q_1 < q_2 < \cdots < q_k =n$ obtained by
putting
\[
q_i = | V_1 | + | V_2 | + \cdots + | V_i |, \ \
1 \leq i \leq k.
\]
It is then immediate that the multilinear functional
$F^{ [ \pi ] } : \cB^{n-1} \to \cB$ acts by
\[
F^{ [ \pi ] } (b_1, \ldots , b_{n-1} ) =
F^{ [q_1] } \bigl( b_1, \ldots , b_{q_1 -1} \bigr) b_{q_1} \times
\]
\[
\times F^{ [q_2 - q_1] }
\bigl( b_{q_1 + 1}, \ldots , b_{q_2 -1} \bigr) b_{q_2} \cdots
b_{q_{k-1}} F^{ [q_k - q_{k-1}] }
\bigl( b_{q_{k-1} + 1}, \ldots , b_{q_k -1} \bigr),
\]
for $b_1, \ldots , b_{n-1} \in \cB$.
\end{remark}

\section{R-transform, B-transform, and convolution powers}
\label{Section:Transforms}

\setcounter{equation}{0}
\setcounter{theorem}{0}

In Section 2 we saw that the map $\mu \mapsto M_{\mu}$ is a bijection
from $\Dist ( \cB )$ onto $\Serie ( \cB )$. In this section we
will review two other important bijections from $\Dist ( \cB )$ onto
$\Serie ( \cB )$, which associate to every $\mu$ its $R$-transform
and its $B$-transform. Both these bijections can be treated
combinatorially by using some bijective self-maps of
$\Serie ( \cB )$ -- one bijection which connects the moment series
$M_{\mu}$ to the $R$-transform $R_{\mu}$, and another bijection which
connects the moment series $M_{\mu}$ to the $B$-transform $B_{\mu}$,
$\mu \in \Dist ( \cB )$. For lack of better names, we
will use the notations $\RtoM$ and respectively $\EtatoM$ for
these two self-maps of $\Serie ( \cB )$. It is useful that $\RtoM$
and $\EtatoM$ can be introduced by explicit summation formulas which
don't make reference to distributions, as follows.

$\ $

\begin{notation}   \label{def:4.2}
For $F \in \Serie ( \cB )$ we denote by $\RtoM \,  (F)$ the
series $G \in \Serie ( \cB )$ with terms defined as follows:
\begin{equation}   \label{eqn:4.21}
G^{ [n] } := \sum_{ \pi \in NC(n) } F^{ [ \pi ] },
\ \ \forall \, n \geq 1.
\end{equation}
\end{notation}

\begin{remark}  \label{rem:4.3}
By making appropriate use of Lemma \ref{lemma:3.8}, one finds
that Equation (\ref{eqn:4.21}) extends to the formula
\begin{equation}   \label{eqn:4.31}
G^{ [ \rho ] } = \sum_{  \begin{array}{c}
{\scriptstyle \pi \in NC(n),} \\
{\scriptstyle \pi \leq \rho}
\end{array}  } \  F^{ [ \pi ] },
\end{equation}
holding for all $n \geq 1$ and $\rho \in NC(n)$. One then
invokes the M\"obius inversion formula for the poset $NC(n)$ in
order to invert (\ref{eqn:4.31}); this leads to the formula
\begin{equation}   \label{eqn:4.32}
F^{ [ \rho ] } = \sum_{  \begin{array}{c}
{\scriptstyle \pi \in NC(n),} \\
{\scriptstyle \pi \leq \rho}
\end{array}  } \  \Moeb ( \pi, \rho ) \, G^{ [ \pi ] },
\end{equation}
holding for $n \geq 1$ and $\rho \in NC(n)$, and where ``Moeb''
stands for the M\"obius function of $NC(n)$ (see e.g. the review
of Moeb made in Lecture 10 of \cite{NS2006}). In particular, the terms
of the series $F$ can be recaptured via the formula
\begin{equation}   \label{eqn:4.33}
F^{ [n] } = \sum_{ \pi \in NC(n) }
\Moeb ( \pi, 1_n ) \, G^{ [ \pi ] }, \ \ \forall \, n \geq 1.
\end{equation}
$ $From here one immediately finds that the map
$\RtoM : \Serie ( \cB ) \to \Serie ( \cB )$ is a bijection,
having for inverse the map $G \mapsto F$ described by
Equation (\ref{eqn:4.33}).
\end{remark}

\begin{definition}   \label{def:4.4}
For every distribution $\mu \in \Dist ( \cB )$, the series
\begin{equation}  \label{eqn:4.41}
R_{\mu} :=  \RtoM^{-1} ( \, M_{\mu} \, ) \in \Serie ( \cB )
\end{equation}
is called the {\em $R$-transform} of $\mu$.
\end{definition}

The construction of the bijection $\EtatoM$ and the definition
of the $B$-transform associated to a distribution
$\mu \in \Dist ( \cB )$ are done in exactly the same way, but
where now instead of $NC(n)$ one uses the smaller poset
$\Int (n)$ of interval partitions.

\begin{notation}   \label{def:4.5}
For $F \in \Serie ( \cB )$ we denote by $\EtatoM \ (F)$ the
series $G \in \Serie ( \cB )$ with terms defined as follows:
\begin{equation}   \label{eqn:4.51}
G^{ [n] } := \sum_{ \pi \in \Int (n) } F^{ [ \pi ] },
\ \ \forall \, n \geq 1.
\end{equation}
\end{notation}

\begin{remark}   \label{rem:4.6}
Exactly as in Remark \ref{rem:4.3}, one sees that Equation
(\ref{eqn:4.51}) extends to the formula
\begin{equation}   \label{eqn:4.61}
G^{ [ \rho ] } = \sum_{  \begin{array}{c}
{\scriptstyle \pi \in \Int (n),} \\
{\scriptstyle \pi \leq \rho}
\end{array}  } \  F^{ [ \pi ] },
\end{equation}
holding for all $n \geq 1$ and $\rho \in \Int (n)$. One then
uses M\"obius inversion in the poset $\Int (n)$
in order to invert (\ref{eqn:4.61}). Since $\Int (n)$ is
isomorphic to a Boolean poset, this inversion process is in fact
quite straightforward, and leads to the formula
\begin{equation}   \label{eqn:4.62}
F^{ [ \rho ] } = \sum_{  \begin{array}{c}
{\scriptstyle \pi \in \Int (n),} \\
{\scriptstyle \pi \leq \rho}
\end{array}  } \  (-1)^{ | \pi | -  | \rho | } \, G^{ [ \pi ] },
\end{equation}
holding for $n \geq 1$ and $\rho \in \Int (n)$. In particular, the
terms of the series $F$ are recaptured from $G$ via the formula
\begin{equation}   \label{eqn:4.63}
F^{ [n] } = \sum_{ \pi \in \Int (n) }
(-1)^{ | \pi | - 1 } \, G^{ [ \pi ] }, \ \ \forall \, n \geq 1.
\end{equation}
In this way it becomes clear that the map
$\EtatoM : \Serie ( \cB ) \to \Serie ( \cB )$ is a bijection,
having for inverse the map $G \mapsto F$ described by
Equation (\ref{eqn:4.63}).
\end{remark}

\begin{definition}   \label{def:4.7}
For every distribution $\mu \in \Dist ( \cB )$, the series
\begin{equation}  \label{eqn:4.71}
B_{\mu} := \EtatoM^{-1} ( \, M_{\mu} \, ) \in \Serie ( \cB )
\end{equation}
is called the {\em $B$-transform} of $\mu$. Note that in many sources
this would be called ``the $\eta$-series of $\mu$'', but we reserve
the letter $\eta$ for maps and covariances.
\end{definition}

Now, with definitions laid out as above, there is no problem to
define generalized convolution powers with respect to
$\boxplus$ and $\uplus$, as follows.

\begin{definition}   \label{def:4.9}
Let $\mu$ be a distribution in $\Dist ( \cB )$, and let
$\alpha : \cB \to \cB$ be a linear map.

$1^o$ We will denote by $\mu^{ \boxplus \alpha }$ the
distribution in $\Dist ( \cB )$ which is uniquely determined by the
fact that its $R$-transform is
\begin{equation}   \label{eqn:4.91}
R_{ \mu^{ \boxplus \alpha } } = \alpha \circ R_{\mu} .
\end{equation}

$2^o$ We will denote by $\mu^{ \uplus \alpha }$ the
distribution in $\Dist ( \cB )$ which is uniquely determined by the
fact that its $B$-series is
\begin{equation}   \label{eqn:4.92}
B_{ \mu^{ \uplus \alpha } } = \alpha \circ B_{\mu} .
\end{equation}
\end{definition}

\begin{remark}   \label{rem:4.10}
Directly from the above definition, we have semigroup
properties for each of the two types of convolution powers. More
precisely: for every $\mu \in \Dist ( \cB )$ and every linear maps
$\alpha , \beta : \cB \to \cB$ we have
\[
\Bigl( \, \mu^{ \boxplus \alpha } \, \Bigr)^{\boxplus \beta}
= \mu^{\boxplus ( \beta \circ \alpha ) }
\mbox{ and }
\Bigl( \, \mu^{ \uplus \alpha } \, \Bigr)^{\uplus \beta}
= \mu^{\uplus ( \beta \circ \alpha ) }.
\]
\end{remark}

\section{The bijection connecting R-transform to B-transform}

\setcounter{equation}{0}
\setcounter{theorem}{0}

We continue to use the framework from the preceding two sections.
We will now examine another bijective self-map of $\Serie ( \cB )$,
which combines the two bijections $R_{\mu} \mapsto M_{\mu}$ and
$B_{\mu} \mapsto M_{\mu}$ discussed in Section 4, and acts by the
prescription that
\[
R_{\mu} \mapsto B_{\mu}, \ \ \mu \in \Dist ( \cB ).
\]
It is useful that this bijection can be
introduced by a direct combinatorial formula, without making
explicit reference to the transforms $R$ and $B$. The direct
formula will be given in Definition \ref{def:5.2}, then the
``$R_{\mu} \mapsto B_{\mu}$'' property will be derived in
Proposition \ref{prop:5.4}.

In order to state Definition \ref{def:5.2}, we first review a few
more details of the combinatorics of $NC(n)$.

\begin{remark}     \label{rem:5.1}
(The partial order $\ll$ in $NC(n)$.)
For $\pi , \rho \in NC(n)$ we will write
``$\pi \ll \rho$'' to mean that $\pi \leq \rho$ and that, in
addition, the following condition is fulfilled:
\begin{equation}     \label{eqn:5.11}
\left\{  \begin{array}{l}
\mbox{For every block $W$ of $\rho$ there exists a block}  \\
\mbox{$V$ of $\pi$ such that $\min (W), \max (W) \in V$.}
\end{array}  \right.
\end{equation}
It is immediately verified that ``$\ll$'' is a partial
order relation on $NC(n)$. It is much coarser than the reverse
refinement order, and differs from it in several respects. In
particular, observe that $( NC(n) , \ll )$ has many maximal
elements: they are precisely the interval partitions, and for
every $\pi \in NC(n)$ there exists a unique interval partition
$\rho$ such that $\pi \ll \rho$. (The blocks
of this unique interval partition $\rho$ are in some sense the
convex hulls of the outer blocks of $\pi$.)

A special role in the subsequent calculation will be played by
the partitions $\pi \in NC(n)$ such that $\pi \ll 1_n$. The
latter inequality means (obvious from the definition) that
$\pi$ has a unique outer block, which will be denoted by $V_o(\pi)$.
\end{remark}

\begin{definition}   \label{def:5.2}
We define a map $\Reta : \Serie ( \cB ) \to \Serie ( \cB )$ in
the following way: for every $F \in \Serie ( \cB )$ we put
$\Reta (F)$ to be the series $G \in \Serie ( \cB )$ with
\begin{equation}    \label{eqn:5.21}
G^{[n]} = \sum_{ \begin{array}{c}
{\scriptstyle \pi \in NC(n)} \\
{\scriptstyle \pi \ll 1_n}
\end{array} } \ F^{[ \, \pi \, ]}, \ \ \forall \, n \geq 1.
\end{equation}
\end{definition}

\begin{remark}   \label{rem:5.3}

$1^{o}$ It is useful to invoke once again Lemma \ref{lemma:3.8} in
order to note that Equation (\ref{eqn:5.21}) extends to the formula
\begin{equation}    \label{eqn:5.31}
G^{[ \rho ]} = \sum_{ \begin{array}{c}
{\scriptstyle \pi \in NC(n)} \\
{\scriptstyle \pi \ll \rho}
\end{array} } \ F^{[ \pi ]},
\end{equation}
holding for every $n \geq 1$ and every $\rho \in NC(n)$.

$2^{o}$ The notation ``$\, \Reta \, $'' used in Definition
\ref{def:5.2} is meant to be suggestive of the fact that we are
dealing with the map which ``converts the $R$-transform into the
$B$-transform'' (as will be proved in the next proposition).
This map and its properties were previously studied in
\cite{BN2008}, \cite{BN2009}, in the framework of multi-variable distributions
over $\bC$. A comment on notation: the papers \cite{BN2008}, \cite{BN2009} use
the fairly widespread name of ``$\eta$-series'' for the $B$-transform
of a distribution $\mu$; as a consequence, the map which connects the
transforms is called there by the more sonorous name of ``Reta'',
rather than $\Reta$.
\end{remark}

\begin{proposition}   \label{prop:5.4}
The maps $\RtoM$, $\EtatoM$ from Section 4 and the map
$\Reta$ from Definition \ref{def:5.2} are related by the formula
\begin{equation}  \label{eqn:5.41}
\EtatoM \circ \Reta = \RtoM .
\end{equation}
As a consequence, it follows that $\Reta$ is a bijection from
$\Serie ( \cB )$ onto itself, and has the property that
\begin{equation}  \label{eqn:5.42}
\Reta ( R_{\mu} ) = B_{\mu},
\ \ \forall \, \mu \in \Dist ( \cB ).
\end{equation}
\end{proposition}

\begin{proof}
For the verification of (\ref{eqn:5.41}) let us consider a series
$F \in \Serie ( \cB )$ and let us make the notations
\[
\Reta (F) =: G, \mbox{ then } \EtatoM (G) =: H.
\]
We have to prove that $H$ is equal to $\RtoM (F)$. In order to
verify this, we pick a positive integer $n$ and we calculate:
\begin{align*}
H^{ [n] }
& = \sum_{ \rho \in \Int (n)} \, G^{ [ \rho ] }
  \ \mbox{ (by the definition of $\EtatoM$)}                   \\
& = \sum_{ \rho \in \Int (n)} \, \Bigl( \,
    \sum_{ \begin{array}{c}
{\scriptstyle \pi \in NC(n)} \\
{\scriptstyle \pi \ll \rho}
\end{array} } \ F^{[ \pi ]} \, \Bigr)
  \ \mbox{ (by Equation (\ref{eqn:5.31})). }
\end{align*}
We next observe that the double sum which has appeared is in fact
just a sum over $\pi \in NC(n)$; this is due to the observation,
recorded in Remark \ref{rem:5.1}, that for every
$\pi \in NC(n)$ there exists a unique $\rho \in \Int (n)$
such that $\pi \ll \rho$. So then our calculation for $H^{ [n] }$
becomes
\[
H^{ [n] } = \sum_{ \pi \in NC(n) } \ F^{ [ \pi ] }
= \Bigl( \, \RtoM (F) \, \Bigr)^{ [n] },
\]
and the equality $\RtoM (F) = H$ follows.

Since we saw in Section 4 that each of the maps
$\RtoM$ and $\EtatoM$ is a bijection from $\Serie ( \cB )$ to
itself, the formula obtained in (\ref{eqn:5.41}) implies that
$\Reta = \EtatoM^{-1} \circ \RtoM$ has this property as well.

Finally, in order to obtain Equation (\ref{eqn:5.42}) we write
\begin{align*}
\EtatoM \Bigl( \, \Reta \, ( R_{\mu} ) \Bigr)
& = \RtoM ( R_{\mu} )
  \ \mbox{ (by (\ref{eqn:5.41})) }                    \\
& = M_{\mu}
  \ \mbox{ (by the definition of $\RtoM$) }           \\
& = \EtatoM ( B_{\mu} )
  \ \mbox{ (by the definition of $\EtatoM$). }
\end{align*}
Since $\EtatoM$ is one-to-one, it follows that
$\Reta ( R_{\mu} ) = B_{\mu}$, as claimed.
\end{proof}

In the discussion about the transformations $\bB_{\alpha}$ of
next section, we will need to extend Definition \ref{def:5.2}
to a family of bijective maps
$\Reta_{\alpha} : \Serie ( \cB ) \to \Serie ( \cB )$, where
$\alpha$ runs in the set of linear transformations from
$\cB$ to $\cB$. The original $\Reta$ from Definition \ref{def:5.2}
will correspond to the special case when $\alpha$ is the identity
transformation of $\cB$. The definition of
$\Reta_{\alpha} (F)$ is obtained by substituting $\alpha \circ F$
instead of $F$ on the right-hand side of Equation (\ref{eqn:5.21}),
but where we make one important exception to this substitution rule:
for every $\pi \in NC(n)$ such that $\pi \ll 1_n$, the unique outer
block of $\pi$ still carries with it a term of the series $F$
(not substituted by the corresponding term of $\alpha \circ F$).
Because of this exception, the formal definition of
$\Reta_{\alpha} (F)$ will thus be phrased in terms of colourings
of non-crossing partitions, as discussed in Remark \ref{rem:3.9}
-- specifically, we will use the colouring of $\pi \ll 1_n$ where
the unique outer block $V_o(\pi)$ of $\pi$ is coloured differently from the
other blocks.

\begin{definition}    \label{def:5.5}

$1^o$ Let $\pi$ be a partition in $NC(n)$ such that
$\pi \ll 1_n$
%and let $W$ be the unique outer block of $\pi$.
We will denote by $\oo_{\pi}$ the colouring of $\pi$ defined by
\[
\oo_{\pi} (V) = \left\{
\begin{array}{lc}
1, & \mbox{if $V=V_o(\pi)$}    \\
2, & \mbox{if $V$ is a block of $\pi$ such that $V \neq V_o(\pi)$.}
\end{array}  \right.
\]

$2^o$ Let $\alpha : \cB \to \cB$ be a linear transformation. We
define a map $\Reta_{\alpha} : \Serie ( \cB ) \to \Serie ( \cB )$
in the following way: for every $F \in \Serie ( \cB )$ we put
$\Reta_{\alpha} (F)$ to be the series $G \in \Serie ( \cB )$ with
\begin{equation}    \label{eqn:5.51}
G^{[n]} = \sum_{ \begin{array}{c}
{\scriptstyle \pi \in NC(n)} \\
{\scriptstyle \pi \ll 1_n}
\end{array} } \ (F, \alpha \circ F)^{[ \, \pi, \, \oo_{\pi} ]},
\ \ \forall \, n \geq 1
\end{equation}
(and where the right-hand side of Equation (\ref{eqn:5.51}) follows
the notations introduced in Remark \ref{rem:3.9}).
\end{definition}

\begin{remark}   \label{rem:5.6}

$1^o$ In order to get a better idea about how $\Reta_{\alpha}$ works,
let us write down explicitly what Equation (\ref{eqn:5.51}) becomes
for some small values of $n$. We have
\[
G^{[1]} = F^{[1]}, \ \ G^{[2]} (b) = F^{[2]} (b),
\]
\[
G^{[3]} (b_1, b_2) = F^{[3]} (b_1, b_2) +
F^{[2]} \bigl( \, b_1 \, \alpha ( \, F^{[1]} \, ) \, b_2 \, \bigr) ,
\]
\begin{align*}
G^{[4]} (b_1, b_2, b_3)
& = F^{[4]} (b_1, b_2, b_3)
  + F^{[3]} \bigl( \, b_1 \, \alpha ( \, F^{[1]} \, ) \, b_2
    , b_3 \, \bigr)                  \\
& + F^{[3]} \bigl( \, b_1 , b_2 \, \alpha ( \, F^{[1]} \, ) \, b_3
    \, \bigr)
  + F^{[2]} \bigl( \, b_1 \, ( \alpha \circ F^{[2]} ) (b_2) \, b_3
    \, \bigr)                        \\
& + F^{[2]} \bigl( \, b_1 \, \alpha ( \, F^{[1]} \, ) \, b_2
    \, \alpha ( \, F^{[1]} \, ) \, b_3 \, \bigr) .
\end{align*}

$2^o$ From the above definitions it is immediate that the map
$\Reta$ from Definition \ref{def:5.2} becomes $\Reta_1$, where
$1: \cB \to \cB$ is the identity. Let us also note that if
we denote by $0: \cB \to \cB$ the map which is identically equal
to $0$, then $\Reta_0$ is the identity map on $\Serie ( \cB )$
(indeed, in the sum on the right-hand side of Equation
(\ref{eqn:5.51}) the only term which survives is the one indexed by
$1_n$, and thus we get $G^{[n]} = F^{[n]}$ for every $n \geq 1$).

$3^o$ Clearly, the definition of $\Reta_{\alpha}$ was made in such
a way that we have
\begin{equation}   \label{eqn:5.61}
\alpha \circ \Reta_{\alpha} (F) = \Reta ( \alpha \circ F),
\ \ \forall \, F \in \Serie ( \cB ).
\end{equation}
In the case when $\alpha$ is invertible, one can thus introduce
$\Reta_{\alpha}$ by the simpler formula
\begin{equation}   \label{eqn:5.62}
\Reta_{\alpha} (F) = \alpha^{-1} \circ \Reta ( \alpha \circ F),
\ \ F \in \Serie ( \cB ).
\end{equation}
The summation formula used in Equation (\ref{eqn:5.51}) is more
tortuous, but has the merit that it works without assuming that
$\alpha$ is invertible.

$4^o$ The main point we want to make about the maps $\Reta_{\alpha}$
is that they form a commutative semigroup under composition. This
is stated precisely in Proposition \ref{prop:5.9} below. In the
proof of Proposition \ref{prop:5.9} we will use an important
property of the partial order $\ll$, reviewed in Proposition
\ref{prop:5.8}, which essentially says that $\ll$ has ``some
Boolean lattice features'' embedded into it.
\end{remark}

\begin{definition}   \label{def:5.7}
Let $\pi, \rho$ be partitions in $NC(n)$ such that
$\pi \ll \rho$. A block $V$ of $\pi$ is said to be
{\em $\rho$-special} when there exists
a block $W$ of $\rho$ such that $\min (V) = \min (W)$ and
$\max (V) = \max (W)$.
\end{definition}

\begin{proposition}   \label{prop:5.8}
Let $\pi \in NC(n)$ be such that $\pi \ll 1_n$, and consider
the set of partitions
\begin{equation}   \label{eqn:5.81}
\{ \rho \in NC(n) \mid \pi \ll \rho \ll 1_n \} .
\end{equation}
Then
$\rho \mapsto \{ V \in \pi \mid V \mbox{ is $\rho$-special} \}$
is a one-to-one map from the set (\ref{eqn:5.81}) to the set of
subsets of $\pi$. The image of this map is equal to
$\{ \fV \subseteq \pi \mid \fV \ni V_o(\pi) \}$.
%, where $V_0$ denotes the unique outer block of $\pi$.
\end{proposition}

\vspace{10pt}

\noindent
For the proof of Proposition \ref{prop:5.8}, the reader is
referred to Proposition 2.13 and Remark 2.14 of \cite{BN2008}.

$\ $

\begin{proposition}    \label{prop:5.9}
For any linear transformations $\alpha, \beta : \cB \to \cB$,
one has that
\begin{equation}   \label{eqn:5.91}
\Reta_{\alpha}  \circ \Reta_{\beta} = \Reta_{\alpha + \beta}.
\end{equation}
\end{proposition}

\begin{proof}  Let $F$ be a series in $\Serie ( \cB )$, and let us
denote $\Reta_{\beta} (F) =: G$, $\Reta_{\alpha} (G) =: H$. We have
to prove that $H = \Reta_{\alpha + \beta} (F)$. For the whole proof
we fix a positive integer $n$, for which we will verify that the
$n$-th term of $H$ is equal to the $n$-th term of
$\Reta_{\alpha + \beta} (F)$.

$ $From the definition of $\Reta_{\beta}$ it follows that we have
\begin{equation}    \label{eqn:5.82}
H^{[n]} =
\sum_{   \begin{array}{c}
{\scriptstyle \rho \in NC(n), }   \\
{\scriptstyle \rho \ll 1_n }
\end{array} } \ ( G, \beta \circ G )^{ [ \rho, \oo_{\rho} ] } .
\end{equation}

Fix for the moment a partition $\rho \in NC(n)$ such that
$\rho \ll 1_n$. Let us write explicitly
$\rho = \{ V_1, \ldots , V_k \}$, where $V_1$ is the block
which contains $1,n$.
Then $( G, \beta \circ G )^{ [ \rho, \oo_{\rho} ] }$ is a
multilinear functional from $\cB^{n-1}$ to $\cB$, and its
explicit descriptions involves the functionals
\[
G^{ [ \ |V_1| \ ] } \mbox{ and }
\beta \circ G^{ [ \ |V_2| \ ] }, \ldots ,
\beta \circ G^{ [ \ |V_k| \ ] }
\]
nested in various ways (with each of these functionals used
exactly once, and with $G_{ |V_1| }$ appearing ``on the outside'').
Let us next replace each of
$G^{ [ \ |V_1| \ ] }, \ldots , G^{ [ \ |V_k| \ ] }$ from how they
are defined (in reference to the terms of $F$ and of
$\alpha \circ F$) in Equation (\ref{eqn:5.51}). This gives us the
functional $( G, \beta \circ G )^{ [ \rho ] }$ expressed as a sum
of the form
\[
( G, \beta \circ G )^{ [ \rho, \oo_{\rho} ] } =
\sum_{   \begin{array}{c}
{\scriptstyle \pi \in NC(n), }   \\
{\scriptstyle \pi \ll \rho }
\end{array} } \ \term_{\pi},
\]
where every functional $\term_{\pi} : \cB^{n-1} \to \cB$ is obtained
by nesting (in the way dictated by the nestings of blocks of $\pi$)
some terms of the series $F, \alpha \circ F$ and $\beta \circ F$.
(It is important to note here that, because of how our definitions
are run, we never get to deal with terms of the functional
$\alpha \circ \beta \circ F$.) A moment's thought shows in fact that
the precise formula for $\term_{\pi}$ is
\[
\term_{\pi} = ( F, \alpha \circ F, \beta \circ F)^{ [ \pi ,
c_{\pi, \rho} ] }
\]
where the colouring $c_{\pi , \rho}$ of $\pi$ goes in the way
described as follows:
%denoting by $V_o$ the unique outer block of $\pi$,
we colour a general block $V$ of $\pi$ by putting
\[
c_{ \pi , \rho } (V) =
\left\{  \begin{array}{ll}
1, & \mbox{ if $V = V_o(\pi)$}                               \\
2, & \mbox{ if $V$ is $\rho$-special but $V \neq V_o(\pi)$}  \\
3, & \mbox{ if $V$ is not $\rho$-special.}
\end{array} \right.
\]

Returning to the formula for $H^{[n]}$, we have thus obtained that
\[
H^{ [n] }
= \sum_{   \begin{array}{c}
{\scriptstyle \rho \in NC(n), }   \\
{\scriptstyle \rho \ll 1_n }
\end{array} } \ \Bigl( \,
\sum_{   \begin{array}{c}
{\scriptstyle \pi \in NC(n), }   \\
{\scriptstyle \pi \ll \rho   }
\end{array} } \
( F, \alpha \circ F, \ \beta \circ F )^{ [ \pi , c_{\pi , \rho} ] }
\ \Bigr) .
\]
Change the order of summation, this becomes
\[
H^{ [n] } = \sum_{   \begin{array}{c}
{\scriptstyle \pi \in NC(n), }   \\
{\scriptstyle \pi \ll 1_n }
\end{array} } \ \Bigl( \,
\sum_{   \begin{array}{c}
{\scriptstyle \rho \in NC(n), }   \\
{\scriptstyle such \ that }       \\
{\scriptstyle \pi \ll \rho \ll 1_n }
\end{array} } \
( F, \alpha \circ F, \ \beta \circ F )^{ [ \pi , c_{\pi , \rho} ] }
\ \Bigr) .
\]
Fix $\pi$ and use the parametrization of
$\{ \rho \in NC(n) \mid \pi \ll \rho \ll 1_n \}$ provided by
Proposition \ref{prop:5.8}. Then group together $\alpha$'s
and $\beta$'s into occurrences of $\alpha + \beta$ -- we arrive
exactly at the description for the $n$-th term of the series
$\Reta_{\alpha + \beta} (F)$.
\end{proof}

\begin{corollary}   \label{cor:5.10}
Let $\alpha : \cB \to \cB$ be a linear transformation. The map
$\Reta_{\alpha} : \Serie ( \cB ) \to \Serie ( \cB )$ is bijective
and has inverse equal to $\Reta_{- \alpha}$.
\end{corollary}

\begin{proof}
This is immediate from Proposition \ref{prop:5.9} and the
fact that $\Reta_0$ is the identity map on $\Serie ( \cB )$.
\end{proof}

\begin{remark}   \label{rem:5.11}
$ $From the above corollary we get in particular an explicit formula
for the inverse of the original bijection $\Reta$ from Definition
\ref{def:5.2}. Indeed, this inverse is
\[
\Reta^{-1} = \Reta_1^{-1} = \Reta_{-1} ,
\]
with $(-1) : \cB \to \cB$ being the map $b \mapsto -b$. But
when we invoke Equation (\ref{eqn:5.62}) in the special case of
$\alpha = -1$, the series $\alpha^{-1} \circ \Reta ( \alpha \circ F )$
from its right-hand side simply becomes $- \Reta (-F)$.
We are thus led to the conclusion that the inverse of $\Reta$
acts by
\begin{equation}    \label{eqn:5.71}
\Reta^{-1} (F) = - \Reta (-F), \ \ \forall \, F \in \Serie ( \cB ).
\end{equation}
\end{remark}

\section{The transformations \boldmath{$\bB_{\alpha}$} }

\setcounter{equation}{0}
\setcounter{theorem}{0}

In this section we continue to use the framework and notations
considered in Sections 3--5.

\begin{definition}    \label{Definition:BP}
The {\em Boolean-to-free Bercovici-Pata bijection} is the map
$\bB : \Dist(\cB) \to \Dist ( \cB )$ defined via the
requirement that
\begin{equation}   \label{eqn:6.11}
R_{\bB[\mu]} = B_\mu,  \ \ \mbox{ for } \mu \in \Dist ( \cB ).
\end{equation}
\end{definition}

\begin{remark}   \label{rem:6.2}
$ $From the discussion in Section 4 it is clear that $\bB$ is indeed 
a bijection from $\Dist ( \cB )$ to itself; indeed, one can write
$\bB = \underline{R}^{-1} \circ \underline{B}$, where the bijections
$\underline{R}, \underline{B} : \Dist ( \cB ) \to \Serie ( \cB )$
are defined by sending $\mu \mapsto R_{\mu}$ and respectively
$\mu \mapsto B_{\mu}$, for $\mu \in \Dist ( \cB )$).
The bijection $\bB$ is important because it has meaning in
analytic framework, where it sends general distributions to
$\boxplus$-infinitely divisible distributions. (In the
$\bC$-valued framework, this was found by Bercovici and Pata
\cite{BP1999}. The $\cB$-valued version of the result was recently
established in \cite{BPV2010}.)

In this section we show how the bijection $\bB$ from Definition
\ref{Definition:BP} is incorporated into a semigroup of bijective
transformations of $\Dist ( \cB )$, defined as follows.
\end{remark}

\begin{definition}  \label{def:6.3}
Let $\mu$ be a distribution in $\Dist ( \cB )$ and let
$\alpha : \cB \to \cB$ be a linear map. We define a new distribution
$\bB_{\alpha} ( \mu ) \in \Dist ( \cB )$ by requiring that its
$R$-transform is
\begin{equation}  \label{eqn:6.31}
R_{ \bB_\alpha ( \mu ) }
=  \Reta_{\alpha} \Bigl( \, R_{\mu} \, \Bigr).
\end{equation}
In this way, for every fixed $\alpha$ we get a map
$\bB_{\alpha} : \Dist ( \cB ) \to \Dist ( \cB )$.
\end{definition}

\begin{theorem}   \label{thm:6.4}

$1^o$ For any two linear transformations
$\alpha, \beta : \cB \to \cB$ one has that
\begin{equation}  \label{eqn:6.41}
\bB_{ \alpha } \circ \bB_{\beta} = \bB_{\alpha + \beta} .
\end{equation}

$2^o$ Let $0: \cB \to \cB$ be the linear transformation which is
identically equal to $0$. Then $\bB_0$ is the identity map on
$\Dist ( \cB )$.

$3^o$ For every linear transformation $\alpha : \cB \to \cB$,
the map $\bB_{\alpha} : \Dist ( \cB ) \to \Dist ( \cB )$ is
bijective, with inverse equal to $\bB_{- \alpha }$.

$4^o$ Let $1: \cB \to \cB$ be the identity transformation.
Then $\bB_1 = \bB$ (the Bercovici-Pata bijection reviewed in
Definition \ref{Definition:BP}).
\end{theorem}

\begin{proof} In order to prove $1^o$, let us fix a
$\mu \in \Dist ( \cB )$ and show that
\begin{equation}  \label{eqn:6.42}
\bB_{\alpha} ( \bB_{\beta} ( \mu ) ) = \bB_{\alpha + \beta} ( \mu ).
\end{equation}
We do this by verifying that the distributions on the two
sides of Equation (\ref{eqn:6.42}) have the same $R$-transform.
Indeed, by starting from the left-hand side we can write
\begin{align*}
R_{  \bB_{\alpha} ( \bB_{\beta} ( \mu ))  }
& = \Reta_{\alpha} \Bigl( \, R_{ \bB_{\beta} ( \mu ) } \, \Bigr)  \\
& = \Reta_{\alpha} \Bigl( \, \Reta_{\beta}
        \bigl( \, R_{\mu} \, \bigr) \, \Bigr)  \\
& = \Reta_{\alpha + \beta} \bigl( \, R_{\mu} \, \bigr)
   \mbox{ (by Proposition \ref{prop:5.9}) }                \\
& = R_{  \bB_{\alpha + \beta} ( \mu )  } .
\end{align*}

Statement $2^o$ is immediate from the fact that
$\Reta_0$ is the identity map on $\Serie ( \cB )$, and
$3^o$ is an immediate consequence of $1^o$ and $2^o$.

Finally, for $4^o$ let us fix a $\mu \in \Dist ( \cB )$
for which we prove that
$\bB_1 ( \mu ) = \bB ( \mu )$.
We do this by verifying that the two distributions in question
have the same $R$-transform:
\begin{align*}
R_{ \bB_1 ( \mu ) }
& = \Reta_1 \Bigl( \, R_{\mu} \, \Bigr)               \\
& = \Reta   \Bigl( \, R_{\mu} \, \Bigr)
   \mbox{ (since $\Reta_1 = \Reta$) }                  \\
& = B_{\mu} \mbox{ (by Proposition \ref{prop:5.4}) }  \\
& = R_{ \bB ( \mu ) }.
\end{align*}
\end{proof}

In Definition \ref{def:6.3}, the distribution $\bB_{\alpha} ( \mu )$
was introduced via a prescription on what is its $R$-transform. We
observe next that we could have equally well made the definition
by a very similar prescription phrased in terms of $B$-transforms.

\begin{proposition}   \label{prop:6.5}
For every $\mu \in \Dist ( \cB )$ and every linear transformation
$\alpha : \cB \to \cB$ one has that
\begin{equation}   \label{eqn:6.51}
B_{ \bB_{\alpha} ( \mu ) } = \Reta_{\alpha}
\bigl( B_{\mu} \bigr).
\end{equation}
\end{proposition}

\begin{proof} We calculate:
\begin{align*}
B_{ \bB_{\alpha} ( \mu ) }
& = R_{ \bB ( \bB_{\alpha} ( \mu )) }    \\
& = R_{ \bB_{\alpha + 1} ( \mu ) }
  \mbox{ (by $1^o$ and $4^o$ of Theorem \ref{thm:6.4}) }      \\
& = \Reta_{\alpha + 1} ( R_{\mu} )
  \mbox{ (by the definition of $\bB_{\alpha + 1}$) }         \\
& = \Reta_{\alpha} \Bigl( \, \Reta ( R_{\mu} ) \, \Bigr)
  \mbox{ (by Proposition \ref{prop:5.9}) }                   \\
& = \Reta_{\alpha} ( B_{\mu} ).
\end{align*}
\end{proof}

Yet another way of approaching the transformations
$\bB_{\alpha}$ can be obtained in terms of $\boxplus$ and
$\uplus$ convolution powers.

\begin{proposition}   \label{prop:6.6}
For every $\mu \in \Dist ( \cB )$ and every linear transformation
$\alpha : \cB \to \cB$ one has that
\begin{equation}   \label{eqn:6.61}
\Bigl( \, \bB_{\alpha} ( \mu ) \, \Bigr)^{\uplus (1 + \alpha )}
= \mu^{\boxplus (1 + \alpha)}.
\end{equation}
As a consequence, if $\alpha : \cB \to \cB$ is a linear
transformation such that $1 + \alpha$ is invertible, then the map
$\bB_{\alpha} : \Dist ( \cB ) \to \Dist ( \cB )$ can be described
by the formula
\begin{equation}  \label{eqn:6.62}
\bB_\alpha ( \mu ) = \Bigl( \, \mu^{\boxplus (\iB + \alpha)} \,
\Bigr)^{\uplus ((\iB + \alpha)^{-1})}, \ \
\mu \in \Dist ( \cB ).
\end{equation}
\end{proposition}

\begin{proof}
We verify that the distributions on the two sides
of (\ref{eqn:6.61}) have the same $B$-transform. We start from
the left-hand side:
\begin{align*}
B_{ ( \bB_{\alpha} ( \mu ) )^{\uplus (1 + \alpha )} }
& = (1 + \alpha ) \circ B_{ ( \bB_{\alpha} ( \mu ) ) }   \\
& = (1 + \alpha ) \circ \Reta_{\alpha}
    \bigl( \, B_{ \mu } \, \bigr)
    \mbox{ (by Proposition \ref{prop:6.5}) }             \\
& = (1 + \alpha ) \circ \Reta_{\alpha}
    \bigl( \, \Reta ( R_{ \mu } ) \, \bigr)              \\
& = (1 + \alpha ) \circ \Reta_{1 + \alpha}
    \bigl( \, R_{ \mu } \, \bigr)                        \\
& = \Reta \bigl( \, (1 + \alpha) \circ R_{\mu} \, \bigr)
    \mbox{ (by Remark \ref{rem:5.6}.3) }                 \\
& = \Reta \Bigl( \, R_{\mu^{ \boxplus (1+ \alpha) }} \, \Bigr) \\
& = B_{  \mu^{ \boxplus (1+ \alpha) }  } .
\end{align*}
\end{proof}

We conclude this section by showing how the transformations
$\bB_{\alpha}$ relate to the $\cB$-valued free Brownian motion.
We do this by putting into evidence a transformation ``$\Phi$''
which is the $\cB$-valued analog for the transformation with
the same role (and same name) that was introduced in the
$\bC$-valued framework in \cite{BN2008b, BN2009}. This
operator-valued version of $\Phi$ is introduced in Definition
\ref{def:6.8} below, on a line similar to the one used in
\cite{BN2009} in the multi-variable $\bC$-valued framework.
Before stating that, we briefly recall here some basic
terminology related to $\cB$-valued semicircular elements and
free Brownian motion.

\begin{definition}
For a linear map $\eta : \cB \to \cB$, we denote by
$\gamma_{\eta}$ the distribution in $\Dist ( \cB )$ which is
uniquely determined by the requirement that its $R$-transform
acts as follows:
\[
R_{\gamma_{\eta}}^{ [2] } = \eta, \mbox{ and }
R_{\gamma_{\eta}}^{ [n] } = 0 \mbox{ for every $n \neq 2$.}
\]
This $\gamma_{\eta}$ is called the {\em $\cB$-valued semicircular
distribution of variance $\eta$}. In the case when $\eta$ is
completely positive, the distribution $\gamma_{\eta}$ belongs
to $\Sigma^{0} ( \cB )$ (see e.g. \cite{Sh1999, S1998}).

For $\mu \in \Dist ( \cB )$, the collection of distributions
\[
\{ \mu \boxplus \gamma_{\eta} \mid \eta : \cB \to \cB,
\mbox{ linear} \}
\]
is sometimes referred to as the ``$\cB$-valued free Brownian
motion started at $\mu$''.
\end{definition}

\begin{definition}  \label{def:6.8}
For $\beta : \cB \langle \cX \rangle \rightarrow \cB$ a $\bC$-linear
map, define $\Phi[\beta] \in \Dist(\cB)$ by prescribing the
$B$-transform of $\Phi [ \beta ]$ to act as follows:
\[
B_{\Phi[\beta]}^{[1]} = 0 \in \cB
\]
and then
\[
B_{\Phi[\beta]}^{[n]}[b_1, b_2, \ldots, b_{n-1}]
= \beta[b_1 \cX \cdots \cX b_{n-1}], \ \ \forall \, n \geq 2,
\ \forall \, b_1, \ldots , b_{n-1} \in \cB .
\]
$\Phi$ is clearly a bijection
\[
\begin{split}
\Phi : \set{\bC\text{-linear } \beta: \cB \langle \cX \rangle \rightarrow \cB}
\rightarrow \set{\text{unital } \cB\text{-bimodule maps } \mu: \cB \langle \cX \rangle \rightarrow \cB \text{ s.t. } \mu |_{\cB \cX \cB} = 0}.
\end{split}
\]
Note that if $\beta \in \Dist(\cB)$, then $B_{\Phi[\beta]}^{[2]}[b_1] = b_1$ and for $n > 2$,
\[
B_{\Phi[\beta]}^{[n]}[b_1, b_2, \ldots, b_{n-1}] = b_1 M_\beta^{[n-2]}[b_2, \ldots ,b_{n-2}] b_{n-1}.
\]
\end{definition}

\begin{theorem}  \label{thm:7.11}
Let $\alpha: \cB \rightarrow \cB$ be a linear transformation,
$\beta \in \Dist(\cB)$, and $\set{\gamma_\alpha}$ the $\cB$-valued
semicircular distribution with covariance $\alpha$. Then
\begin{equation}  \label{eqn:7.111}
\Phi[\beta \boxplus \gamma_\alpha] = \bB_\alpha[\Phi[\beta]].
\end{equation}
\end{theorem}

\begin{proof} We prove that the distributions on the two sides of
Equation (\ref{eqn:7.111}) have the same $B$-transform.
For $n = 1$,
\[
B_{\Phi[\beta \boxplus \gamma_\alpha]}^{[1]}
= 0
= \Reta_\alpha[B_{\Phi[\beta]}]^{[1]}
= B_{\bB_\alpha[\Phi[\beta]]}^{[1]}.
\]
For $n \geq 2$,
\[
\begin{split}
B_{\Phi[\beta \boxplus \gamma_\alpha]}^{[n]}[b_1, b_2, \ldots, b_{n-1}]
& = b_1 M_{\beta \boxplus \gamma_\alpha}^{[n-2]}(b_2, \ldots, b_{n-2}) b_{n-1} \\
& = \sum_{\pi \in NC(n-2)} \sum_{\substack{S \subset \pi \\ V \in S \Rightarrow \abs{V} = 2}} b_1 (\alpha, R_\beta)^{[\pi, c_{\pi, S, 1}]}(b_2, \ldots, b_{n-2}) b_{n-1} \\
& = \sum_{\substack{\pi \ll 1_n \\ \abs{V_o(\pi)} = 2}} \
    \sum_{\substack{V_o(\pi) \in S \subset \pi \\ V \in S \Rightarrow \abs{V} = 2}} (1, \alpha, R_\beta)^{[\pi, c_{\pi, S, 2}]}(b_1, b_2, \ldots, b_{n-2}, b_{n-1}),
\end{split}
\]
where
\[
c_{\pi, S, 1} =
\begin{cases}
1, & \text{if } V \in S \\
2, & \text{if } V \not \in S
\end{cases}
\]
and
\[
c_{\pi, S, 2} =
\begin{cases}
1, & \text{if } V = V_o(\pi) \\
2, & \text{if } V \in S \text{ but } V \neq V_o(\pi) \\
3, & \text{if } V \not \in S.
\end{cases}
\]
On the other hand,
\[
\begin{split}
B_{\bB_\alpha[\Phi[\beta]]}^{[n]}[b_1, b_2, \ldots, b_{n-1}]
& = \Reta_\alpha(B_{\Phi[\beta]})^{[n]}[b_1, b_2, \ldots, b_{n-1}] \\
& = \sum_{\rho \ll 1_n} (B_{\Phi[\beta]}, \alpha \circ B_{\Phi[\beta]})^{[\rho, \oo_\rho]}[b_1, b_2, \ldots, b_{n-1}].
\end{split}
\]
Note that $B_{\Phi[\beta]}^{[1]} = 0$, so the sum above can be restricted to $\rho$ each of whose classes contains at least $2$ elements.

Given $\pi$ and $S = \set{V_1, \ldots, V_k}$ as above, we define a partition
\[
f(\pi, S) = \rho = (U_1, \ldots, U_k) \ll 1_n
\]
as follows: $V_i \subset U_i$, and $a \in U_i$ if for some $b, c \in V_i$, $b \leq a \leq c$, and for any $j \neq i$ and $d, e \in V_j$, $d \leq a \leq e$ implies $d < b \leq c < e$. Conversely, given $\rho = (U_1, \ldots, U_k)$ each of whose classes contains at least $2$ elements, $V_i = \set{\min(U_i), \max(U_i)}$, and
\[
\pi = \widehat{\pi_1} \cup \cdots \cup \widehat{\pi_k},
\]
(in the notation of Lemma~\ref{lemma:3.8}), where $\widehat{\pi_i} \ll 1_{U_i}$, $\abs{V_o(\widehat{\pi_i})} = 2$. See also Proposition 5.4, Theorem 6.2 and Remark 6.3 of \cite{BN2009}, or Theorem 11 of \cite{A2010}. To finish the proof, it suffices to show that for each $\rho \ll 1_n$,
\[
\begin{split}
(B_{\Phi[\beta]}, \alpha \circ B_{\Phi[\beta]})^{[\rho, \oo_\rho]}[b_1, b_2, \ldots, b_{n-1}]
& = \sum_{(\pi, S) \in f^{-1}(\rho)} (1, \alpha, R_\beta)^{[\pi, c_{\pi, S, 2}]}(b_1, b_2, \ldots, b_{n-2}, b_{n-1}) \\
& = \sum_{\substack{\pi = \widehat{\pi_1} \cup \cdots \cup \widehat{\pi_k} \\ \widehat{\pi_i} \ll 1_{U_i} \\ \abs{V_o(\widehat{\pi_i})} = 2}} (1, \alpha, R_\beta)^{[\pi, c_{\pi, S, 2}]}(b_1, b_2, \ldots, b_{n-2}, b_{n-1}).
\end{split}
\]
Since, subject to these conditions, the $\set{\widehat{\pi_i}}$ can be chosen independently, it suffices to prove this equality for each class $U_i \in \rho$ separately. If $U_i \neq V_o(\rho)$, the expression is
\[
\begin{split}
(\alpha \circ B_{\Phi[\beta]})^{[\abs{U_i}]}[b_1, b_2, \ldots, b_{\abs{U_i}-1}]
& = \alpha[\beta[b_1 \cX b_2 \cX \ldots \cX b_{\abs{U_i}-1}] \\
& = \sum_{\substack{\widehat{\pi_i} \ll 1_{U_i} \\ \abs{V_o(\widehat{\pi_i})} = 2}} (\alpha, R_\beta)^{[\pi, \oo_{\widehat{\pi_i}}]}(b_1, b_2, \ldots, b_{n-2}, b_{\abs{U_i}-1}),
\end{split}
\]
while for $U_i = V_o(\rho)$ the expression is the same without $\alpha$. The result follows.
\end{proof}

$\ $

$\ $

\section{Operator models}

\setcounter{equation}{0}
\setcounter{theorem}{0}

In the remainder of the paper, we will denote by $\cCP(\cB)$ the
space of completely positive $\bC$-linear maps on $\cB$.

\begin{construction}
\label{Construction:lambda-beta}
Let $\lambda \in \cB$ be symmetric, and $\beta: \cB \langle \cX \rangle \rightarrow \cB$ be a $\bC$-linear, completely positive map. Define the $\cB$-valued inner product on $\cB_0 \langle \cX \rangle$ by
\begin{equation}
\label{Beta-inner-product}
\ip{b_0 \cX \ldots \cX b_n}{c_0 \cX \ldots \cX c_k}_\beta
= c_k^\ast \beta[c_{k-1}^\ast \cX \ldots \cX c_0^\ast b_0 \cX \ldots \cX b_{n-1}] b_n.
\end{equation}
Note that a general element of $\cB_0 \langle \cX \rangle$ is of the form $\sum_{i=1}^n P_i \cX b_i$, with $P_i \in \cB \langle \cX \rangle$, and
\[
\ip{\sum_{i=1}^n P_i \cX b_i}{\sum_{i=1}^n P_i \cX b_i}_\beta
= \sum_{i, j=1}^n b_i^\ast \beta[P_i^\ast P_j] b_j \geq 0
\]
since $\beta$ is completely positive. It follows that this inner product is positive. It may be degenerate; however, we will only use this construction for combinatorial computations of moments and cumulants.
\end{construction}

\begin{lemma}
\label{Lemma:Boolean-Fock}
Consider the vector space $\cB \langle \cX \rangle = \cB \oplus \cB_0 \langle \cX \rangle$ with the $\cB$-valued inner product
\[
\ip{b \oplus m}{b' \oplus m'} = (b')^\ast b + \ip{m}{m'}_\beta.
\]
On this vector space, we define maps
\[
a^\ast (b \oplus b_0 \cX \ldots \cX b_n) = 0 \oplus \cX b,
\]
\[
p (b \oplus b_0 \cX \ldots \cX b_n) = 0 \oplus \cX b_0 \cX \ldots \cX b_n,
\]
and
\[
a (b \oplus b_0 \cX \ldots \cX b_n) = \beta[b_0 \cX \ldots \cX b_{n-1}] b_n \oplus 0,
\]
in particular
\[
a (b \oplus b_0 \cX) = \beta[b_0] \oplus 0.
\]
Then $p$ and $a^\ast + a$ are symmetric. Therefore the operator
\[
X = a^\ast + a + p + \lambda
\]
is also symmetric. It follows that
\[
\mu_{(\lambda, \beta)} : \cB \langle \cX \rangle \rightarrow \cB,
\]
defined via
\[
\mu_{(\lambda, \beta)}[b_0 \cX \ldots \cX b_n] = \ip{(b_0 X \ldots X b_n) (1 \oplus 0)}{1 \oplus 0}
\]
is a conditional expectation, $\mu_{(\lambda, \beta)} \in \Sigma(\cB)$. Moreover, if $\beta$ satisfies the boundedness condition \eqref{Eq:Bounded}, then $\mu_{(\lambda, \beta)} \in \Sigma^0(\cB)$.
\end{lemma}

\begin{proof}
We check the first statement. Indeed,
\[
\begin{split}
\ip{a^\ast (b \oplus b_0 \cX \ldots \cX b_n)}{c \oplus c_0 \cX \ldots \cX c_k}
& = \ip{0 \oplus \cX b}{c \oplus c_0 \cX \ldots \cX c_k} \\
& = c_k^\ast \beta[c_{k-1}^\ast \cX \ldots \cX c_0^\ast] b \\
& = \ip{b \oplus 0}{a(c \oplus c_0 \cX \ldots \cX c_k)} \\
& = \ip{b \oplus b_0 \cX \ldots \cX b_n}{a(c \oplus c_0 \cX \ldots \cX c_k)}
\end{split}
\]
and
\[
\begin{split}
\ip{p (b \oplus b_0 \cX \ldots \cX b_n)}{c \oplus c_0 \cX \ldots \cX c_k}
& = \ip{0 \oplus \cX b_0 \cX \ldots \cX b_n}{c \oplus c_0 \cX \ldots \cX c_k} \\
& = c_k^\ast \beta[c_{k-1}^\ast \cX \ldots \cX c_0^\ast \cX b_0 \cX \ldots \cX b_{n-1}] b_n \\
& = \ip{b \oplus b_0 \cX \ldots \cX b_n}{0 \oplus \cX c_0 \cX \ldots \cX c_k} \\
& = \ip{b \oplus b_0 \cX \ldots \cX b_n}{p (c \oplus c_0 \cX \ldots \cX c_k)}.
\end{split}
\]
It follows that $\mu_{(\lambda, \beta)} \in \Sigma(\cB)$. Finally, if $\beta$ satisfies \eqref{Eq:Bounded} with constant $M$ and $\norm{\lambda} \leq M$, then from equation~\eqref{Boolean-sum} below,
\[
\norm{\mu_{(\lambda, \beta)}[\cX b_1 \cX \ldots b_{n-1} \cX]} \leq 2^n M^n \norm{b_1} \cdot \norm{b_2} \cdot \ldots \cdot \norm{b_{n-1}}.
\]
\end{proof}

\begin{lemma}
The Boolean cumulant functionals $B_{(\lambda, \beta)}$ of $\mu_{(\lambda, \beta)}$ are
\[
B_{(\lambda, \beta)}^{[1]} = \lambda
\]
and
\[
B_{(\lambda, \beta)}^{[n]}(b_1, b_2, \ldots, b_{n-1}) = \beta[b_1 \cX \ldots \cX b_{n-1}].
\]
\end{lemma}

\begin{proof}
Using the definition $X = a^\ast + a + p + \lambda$ and the definitions of $a^\ast, a, p, \lambda$,
\begin{equation}
\label{Boolean-sum}
\begin{split}
\mu_{(\lambda, \beta)}[\cX b_1 \ldots b_{n-1} \cX]
& = \ip{(X b_1 \ldots b_{n-1}X) (1 \oplus 0)}{1 \oplus 0} \\
& = \sum_{k=1}^n \sum_{1 \leq i_1 < i_2 < \ldots < i_k = n}
\beta[b_1 \cX \ldots b_{i_1 - 1}] b_{i_1} \beta[b_{i_1 + 1} \cX \ldots b_{i_2 - 1}] b_{i_2} \\
&\qquad\qquad\qquad\qquad \ldots b_{i_{k-1}} \beta[b_{i_{k-1} + 1} \cX \ldots b_{n-1}],
\end{split}
\end{equation}
where $\beta[\emptyset] = \lambda$. On the other hand, combining the last formula in Remark~\ref{rem:3.7}, Notation~\ref{def:4.5}, and Definition~\ref{def:4.7}, we get
\[
\begin{split}
M_{\mu}^{[n]}(b_1, \ldots, b_{n-1})
& = \sum_{k=1}^n \sum_{1 \leq i_1 < i_2 < \ldots < i_k = n}
B_\mu(b_1, \ldots, b_{i_1 - 1}) b_{i_1} B_\mu(b_{i_1 + 1}, \ldots, b_{i_2 - 1}) b_{i_2} \\
&\qquad\qquad\qquad\qquad \ldots b_{i_{k-1}} B_\mu(b_{i_{k-1} + 1}, \ldots, b_{n-1}).
\end{split}
\]
Comparing these two formulas, we get the result.
\end{proof}

Compare with Theorem~5.6 of \cite{PV2010}.

The following result follows from Lemma 2.9 and Theorem 2.5 of \cite{PV2010}, but for completeness we provide a shorter proof for our case.

\begin{lemma}
Let $\mu \in \Sigma(\cB)$. Then for some symmetric element $\lambda \in \cB$ and a $\bC$-linear, completely positive map $\beta : \cB \langle \cX \rangle \rightarrow \cB$, we have
\[
\mu = \mu_{(\lambda, \beta)}.
\]
If $\mu \in \Sigma^0(\cB)$, then $\beta$ satisfies condition \eqref{Eq:Bounded}.
\end{lemma}

\begin{proof}
Put on $\cB \langle \cX \rangle$ the $\cB$-valued inner product
\[
\ip{b_0 \cX \ldots \cX b_n}{c_0 \cX \ldots \cX c_k}_\mu
= \mu[c_k^\ast \cX c_{k-1}^\ast \cX \ldots \cX c_0^\ast b_0 \cX \ldots \cX b_{n-1} \cX b_n],
\]
and denote by $\cH$ the corresponding $\cB$-inner product bimodule. We will identify elements of $\cB \langle \cX \rangle$ with corresponding operators acting on $\cH$ on the left. Denote
\[
P : \zeta \mapsto \mu[\zeta]
\]
the orthogonal projection from $\cH$ onto $\cB \subset \cH$, and let $P^\perp = I - P$, so that
\[
P^\perp : \zeta \mapsto \zeta - \mu[\zeta].
\]
We write
\[
\xi = P^\perp \cX \cdot 1 = \cX - \mu[\cX] \in \cH.
\]
Finally, denote $T = P^\perp \cX P^\perp$; clearly $T$ is a symmetric operator on $\cH$.

Now using the fact that $P$ and $P^\perp$ commute with $\cB$ (in their actions on $\cH$), and
\[
\ip{\chi P \zeta \cdot 1}{1}_\mu = \mu[\chi] \mu[\zeta],
\]
we compute
\[
\begin{split}
\mu[b_0 \cX b_1 \ldots b_{n-1} \cX b_n]
& = b_0 \ip{\cX b_1 \ldots b_{n-1} \cX \cdot 1}{1}_\mu b_n \\
& = b_0 \ip{\cX b_1 (P + P^\perp) \cX b_2 \ldots b_{n-2} (P + P^\perp) \cX b_{n-1} (P + P^\perp) \cX \cdot 1}{1}_\mu b_n \\
& = \sum_{k=1}^{n} \sum_{1 \leq i_1 < i_2 < \ldots < i_k = n} b_0 \ip{\cX b_1 P^\perp \cX \ldots P^\perp \cX b_{i_1 - 1} P^\perp \cX \cdot 1}{1}_\mu b_{i_1} \\
&\qquad \ldots b_{i_{k-1}} \ip{\cX b_{i_{k-1} + 1} P^\perp \cX \ldots P^\perp \cX b_{n-1} P^\perp \cX \cdot 1}{1}_\mu b_n \\
& = \sum_{k=1}^{n} \sum_{1 \leq i_1 < i_2 < \ldots < i_k = n} b_0 \ip{b_1 T b_2 \ldots T b_{i_1 - 1} \xi}{\xi}_\mu b_{i_1} \\
&\qquad \ldots b_{i_{k-1}} \ip{b_{i_{k-1} + 1} T \ldots T b_{n-1} \xi}{\xi}_\mu b_n.
\end{split}
\]
Note that we have used the bimodule property of the $\mu$-inner product. Comparing with formula \eqref{Boolean-sum}, we see that $\mu = \mu_{(\lambda, \beta)}$ for
\[
\lambda = \ip{\cX \cdot 1}{1}_\mu = \mu[\cX],
\]
and
\[
\beta[b_1 \cX b_2 \ldots \cX b_n] = \ip{b_1 T b_2 \ldots T b_n \xi}{\xi}_\mu.
\]
Clearly $\lambda$ is symmetric, and since $T$ is symmetric, $\beta$ is positive. In fact,
\[
\begin{split}
& \sum_{i, j=1}^n c_i^\ast \beta[(b_{i,1} \cX b_{i,2} \ldots \cX b_{i, k(i)})^\ast (b_{j,1} \cX b_{j,2} \ldots \cX b_{j, k(j)})] c_j \\
&\qquad = \sum_{i,j=1}^n \ip{(b_{i,1} T b_{i,2} \ldots T b_{i, k(i)})^\ast (b_{j,1} T b_{j,2} \ldots T b_{j, k(j)}) \xi c_j}{\xi c_i}_\mu \\
&\qquad = \ip{\sum_{i=1}^n (b_{i,1} T b_{i,2} \ldots T b_{i, k(i)}) \xi c_i}{\sum_{i=1}^n (b_{i,1} T b_{i,2} \ldots T b_{i, k(i)}) \xi c_i}_\mu \geq 0,
\end{split}
\]
so $\beta$ is completely positive. If $\mu \in \Sigma^0(\cB)$, then, since $\norm{T} = \norm{P^\perp \cX P^\perp} \leq \norm{\cX}$, $\beta$ satisfies condition \eqref{Eq:Bounded}.
\end{proof}

\begin{theorem}  \label{thm:7.5}
For any $\mu \in \Sigma(\cB)$ and $\alpha \in \cCP(\cB)$, the functional $\mu^{\uplus \alpha}\in \Sigma(\cB)$. Moreover, if $\mu \in \Sigma^0(\cB)$, so is $\mu^{\uplus \alpha}$.
\end{theorem}

\begin{proof}
By the preceding lemma, $\mu = \mu_{(\lambda, \beta)}$ for some $(\lambda, \beta)$. Moreover, replacing the pair $(\lambda, \beta)$ with the pair $(\alpha[\lambda], \alpha \circ \beta)$ gives another pair of the same type, so that for $\mu_{(\alpha[\lambda], \alpha \circ \beta)}$,
\[
B_{(\alpha[\lambda], \alpha \circ \beta)}^{[n]}(b_1, b_2, \ldots, b_{n-1}) = \alpha\left[ B_{(\lambda, \beta)}^{[n]}(b_1, b_2, \ldots, b_{n-1}) \right]
\]
for any $n \geq 1$. So by Definition~\ref{def:4.9},
\[
\mu_{(\alpha[\lambda], \alpha \circ \beta)} = \mu_{(\lambda, \beta)}^{\uplus \alpha}.
\]
The statement about $\Sigma^0(\cB)$ also follows from the preceding lemmas.
\end{proof}

\begin{corollary}
For any freely infinitely divisible $\mu \in \Sigma(\cB)$ and any $\alpha \in \cCP(\cB)$, the functional $\mu^{\boxplus \alpha}\in \Sigma(\cB)$. Moreover, if $\mu \in \Sigma^0(\cB)$, so is $\mu^{\boxplus \alpha}$.
\end{corollary}

\begin{proof}
Theorem~3.4 of \cite{BPV2010} shows that the Bercovici-Pata bijection $\bB$ from Definition~\ref{Definition:BP} maps $\Sigma(\cB)$ bijectively onto freely infinitely divisible elements of $\Sigma(\cB)$, $\Sigma^0(\cB)$ bijectively onto freely infinitely divisible elements of $\Sigma^0(\cB)$, and intertwines $\uplus$ and $\boxplus$. So the result follows from the preceding proposition.
\end{proof}

\begin{remark}
\label{remark:free-Fock}
%If necessary, with the notation from Remark~\ref{rem:3.4}, the free construction is also be easy to do directly,
The preceding corollary can also be proved directly. The full Fock module over $\cB_0 \langle \cX \rangle$ is the $\cB$-bimodule
\[
\cF(\cB_0 \langle \cX \rangle) = \bigoplus_{k=0}^\infty \cB_0 \langle \cX \rangle^{\otimes_{\cB}^k} = \cB \oplus \cB_0 \langle \cX \rangle \oplus (\cB_0 \langle \cX \rangle \otimes_\cB \cB_0 \langle \cX \rangle) \oplus \ldots.
\]
Given $\lambda$ and $\beta$ as in Construction~\ref{Construction:lambda-beta}, on $\cF(\cB_0 \langle \cX \rangle)$, define the $\cB$-valued inner product by
\[
\ip{\xi_1 \otimes \ldots \otimes \xi_k}{\zeta_1 \otimes \ldots \otimes \zeta_n}
= \delta_{kn} \ip{\xi_n}{\ip{\xi_{n-1}}{\ldots \ip{\xi_1}{\zeta_1}_\beta \ldots\zeta_{n-1}}_\beta \zeta_n}_\beta,
\]
where $\ip{\cdot}{\cdot}_\beta$ is given in formula \eqref{Beta-inner-product}. Again, positivity of the inner product follows from complete positivity of $\beta$; compare with Definition~4.6.5 from \cite{S1998}.
%\[
%\ip{b_{1,0} \cX \ldots \cX b_{1, i(1)} \otimes \ldots \otimes b_{k,0} \cX \ldots \cX b_{k, i(k)}}{c_{1,0} \cX \ldots \cX c_{1, j(1)} \otimes \ldots \otimes c_{k,0} \cX \ldots \cX c_{k, j(k)}}
%= b_{k, i(k)}^\ast \beta[b_{k, i(k)-1}^\ast \cX \ldots \cX b_{k,0}^\ast b_{k-1, i(k-1)}^\ast \beta[ \ldots b_{1, i(1)}^\ast \beta[b_{1, i(1)-1}^\ast \cX \ldots \cX b_{1,0}^\ast c_{1,0} \cX \ldots \cX c_{1, j(1)-1}] c_{1, j(1)} \ldots ] c_{k-1, j(k-1)} c_{k,0} \cX \ldots \cX b_{k, j(k)-1}] c_{k, j(k)}.
%\]
On $\cF(\cB_0 \langle \cX \rangle)$, we define maps
\[
a^\ast (\xi_1 \otimes \ldots \otimes \xi_n) = \cX \otimes \xi_1 \otimes \ldots \otimes \xi_n,
\]
\[
p (\xi_1 \otimes \ldots \otimes \xi_n) = \cX \xi_1 \otimes \ldots \otimes \xi_n,
\]
and
\[
a (\xi_1 \otimes \ldots \otimes \xi_n) = \ip{\xi_1}{\cX}_\beta \xi_2 \otimes \ldots \otimes \xi_n.
\]
Then as in Lemma~\ref{Lemma:Boolean-Fock}, $p$, $a^\ast + a$, and
\[
X = a^\ast + a + p + \lambda
\]
are symmetric, and the functional $\mu$ defined via
\[
\mu[b_0 \cX \ldots \cX b_n] = \ip{(b_0 X \ldots X b_n) 1}{1}
\]
is a conditional expectation. By definition (cf. Section 4.7 of \cite{S1998} or Section 3 of \cite{PV2010}), the conditional expectations arising in this construction are precisely all the freely infinitely divisible ones. Moreover, this $\mu = \bB[\mu_{(\lambda, \beta)}]$, and the conditional expectation arising from $(\alpha[\lambda], \alpha \circ \beta)$ is $\mu^{\boxplus \alpha}$.
\end{remark}

\begin{theorem}   \label{thm:7.8}
For each $\alpha \in \cCP(\cB)$, $\bB_\alpha$ maps $\Sigma(\cB)$
to itself.
\end{theorem}

\begin{proof}
Let $\mu \in \Sigma(\cB)$, so that $\mu = \mu_{(\lambda, \beta)}$ for some $\lambda, \beta$ as in Construction~\ref{Construction:lambda-beta}. Modify the construction in Remark~\ref{remark:free-Fock} as follows. On $\cF(\cB_0 \langle \cX \rangle)$, define the $\cB$-valued inner product by
\[
\ip{\xi_1 \otimes \ldots \otimes \xi_k}{\zeta_1 \otimes \ldots \otimes \zeta_n}
= \delta_{kn} \ip{\xi_n}{\ip{\xi_{n-1}}{\ldots \ip{\xi_1}{\zeta_1}_{\alpha \circ \beta} \ldots\zeta_{n-1}}_{\alpha \circ \beta} \zeta_n}_\beta,
\]
in particular
\[
\ip{\xi}{\zeta} = \ip{\xi}{\zeta}_\beta.
\]
Keep $a^\ast$, $p$, and $a(\xi)$ the same, and let
\[
a (\xi_1 \otimes \ldots \otimes \xi_n) = \ip{\xi_1}{\cX}_{\alpha \circ \beta} \xi_2 \otimes \ldots \otimes \xi_n
\]
for $n \geq 2$. Also, let $L(\xi) = \lambda \xi$ while $L(\xi_1 \otimes \ldots \otimes \xi_n) = \alpha[\lambda] \xi_1 \otimes \ldots \otimes \xi_n$ for $n \geq 2$. Then again, $X = a^\ast + a + p + L$ is symmetric, and the functional $\mu_\alpha$ defined via
\[
\mu_\alpha[b_0 \cX \ldots \cX b_n] = \ip{(b_0 X \ldots X b_n) 1}{1}
\]
is in $\Sigma(\cB)$. But now it is easy to check that $B_{\mu_\alpha} = \Reta_\alpha(B_\mu)$, which implies that $\mu_\alpha = \bB_\alpha(\mu)$.
\end{proof}

\begin{theorem}
\label{Thm:alpha-geq-1}
Let $\mu \in \Sigma(\cB)$ and $\alpha \in \cCP(\cB)$ such that $\alpha - \iB$ is completely positive.
Then the functional $\mu^{\boxplus \alpha} \in \Sigma(\cB)$. If $\mu \in \Sigma^0(\cB)$, so is $\mu^{\boxplus \alpha}$.
\end{theorem}

\begin{proof}
The left-hand-side of the formula \eqref{eqn:6.61} in Proposition~\ref{prop:6.6} is well-defined and positive by the preceding proposition, therefore so is the right-hand-side. The boundedness also follows.
\end{proof}

\begin{remark}
Note that we assume $\mu[1] = 1$, and $\mu$ is a $\cB$-bimodule map, so the restriction $\mu |_\cB$ is the identity map, and $\mu$ is an analog of a probability measure. On the other hand, $\beta$ is not necessarily a $\cB$-bimodule map, and the restriction $\beta |_\cB$ is a general completely positive map. So $\beta$ is an analog of a general finite measure.
\end{remark}

\begin{corollary}   \label{cor:7.11}
For completely positive $\beta$,
\[
\Phi[\beta] = \mu_{(0, \beta)}.
\]
So $\Phi$ is also a bijection
\[
\begin{split}
\Phi : \set{\bC\text{-linear, completely positive } \beta: \cB \langle \cX \rangle \rightarrow \cB}
\rightarrow \set{\mu \in \Sigma(\cB) \text{ s.t. } \mu |_{\cB \cX \cB} = 0}.
\end{split}
\]
with a restriction to $\beta$ satisfying \eqref{Eq:Bounded} and $\mu \in \Sigma^0(\cB)$. It also restricts to a bijection
\begin{multline*}
\Phi : \set{\bC\text{-linear, completely positive } \beta: \cB \langle \cX \rangle \rightarrow \cB \text{ s.t. } \beta |_\cB = I} \\
\rightarrow \set{\mu \in \Sigma(\cB) \text{ s.t. } \mu |_{\cB \cX \cB} = 0, \mu[\cX b \cX] = b}.
\end{multline*}
%What if $\beta$ is a bimodule map?
\end{corollary}

\begin{remark}
For $\beta$ completely positive but not necessarily a $\cB$-bimodule or a unital map, $\bB_\alpha[\Phi[\beta]] \in \Sigma(\cB)$ and is centered, so by the preceding corollary it is in the image of $\Phi$. Therefore following
Theorem~\ref{thm:7.11}, one can \emph{define}
\[
\beta \mapsto \beta \boxplus \gamma_\alpha = (\Phi^{-1} \circ \bB_\alpha \circ \Phi)[\beta],
\]
and this transformation preserved complete positivity. One can define the same extension of the free convolution operation using combinatorics, but in that case positivity is unclear.
\end{remark}

\section{Analytic aspects: analytic subordination and the
operator-valued inviscid Burgers equation}

\setcounter{equation}{0}
\setcounter{theorem}{0}

This section is dedicated to a brief outline of some analytic
consequences and aspects of our previous results.

\subsection{The operator-valued analogue of the free heat equation}
One of the fundamental results of Voiculescu was finding
the free analogue of the heat equation. If $X=X^*$ is free
from the centered semicircular random variable $S$ of variance
one, then
$$
\frac{\partial G(t,z)}{\partial t}+G(t,z)\frac{\partial G(t,z)}{
\partial z}=0,\quad \Im z>0,t>0,
$$
where
$$
G(t,z)=G_{X+\sqrt{t}S}(z)=\int_\mathbb R\frac{1}{z-x}\,d\mu_{X+
\sqrt{t}S}.
$$
We shall naturally extend this to the case when $X$ and $S$ are
free over $\mathcal B$, the $\mathcal B$-valued centered semicircular
random variable $S$ has variance $\eta$ and its evolution is as before
according to (completely) positive maps $\rho\colon\mathcal B\to
\mathcal B$.
%%%%%%%%%%%%%%%%%%%%%%%%%%%%%%%%%%%%%%%%%%%%%%%%%%%%%%%%%%%%%%%%%%%
%%%%%%%S: I think it's easier to treat these two things separately,
%%%%%%%what do you say?
%%%%%%%%%%%%%%%%%%%%%%%%%%%%%%%%%%%%%%%%%%%%%%%%%%%%%%%%%%%%%%%%%%%
%In fact, following the ideas from \cite{A2010}
%[is this the one you mean?],
%\cite{BN2008,BN2009}, we shall approach this through the connection
%between $\mathbb B$ and $\Phi$.

We shall consider maps
$\mathcal{CP}(\mathcal B)\times\mathcal B\to\mathcal B$
analytic on some open set in the second coordinate, and G\^{a}teaux
differerentiable in the first.
\begin{proposition}\label{prop:8.1}
Assume that the map $h$ satisfies
\begin{equation}\label{eq:8.1}
h(\eta,b)=h_0(b+\eta(h(\eta,b))),\quad \eta\in\mathcal{CP}(\mathcal B),
\Im b>0.
\end{equation}
If $h_0$ is analytic on the set $\{b\in\mathcal B\colon\Im b>0\}$
and $h(\eta,b)\colon\mathcal{CP}(\mathcal B)\times
\{b\in\mathcal B\colon\Im b>0\}\to\{b\in\mathcal B\colon\Im b>0\}$,
then the following equation is satisfied:
$$
\frac{\partial h(\eta,b)}{\partial\eta}(\rho)-
\frac{\partial h(\eta,b)}{\partial b}(\rho(h(\eta,b)))=0,
$$
where $\rho\in\mathcal{CP}(\mathcal B)$, $\eta\in{\rm Int}(
\mathcal{CP}(\mathcal B))$, and $\Im b>0$. The derivative
with respect to $b$ is in the Fr\'echet sense and the derivative
with respect to $\eta$ is taken in the G\^{a}teaux
sense.
\end{proposition}
\begin{proof}
The proof simply consists in applying the corresponding definitions
and the ``chain rule.''
Let $\eta$ and $\rho$ be as above.
Strictly for convenience, we shall write $\omega=\omega(\eta,b)=
b+\eta(h(\eta,b))$ and express the above in terms of $\omega$
as
$$
\omega(\eta,b)=b+\eta(h_0(\omega(\eta,b))),\quad \eta\in\mathcal{CP}(\mathcal B),
\Im b>0.
$$

Then
\begin{equation}\label{gateaux}
\lim_{t\to0}\frac{\omega(\eta+t\rho,b)-\omega(\eta,b)}{t}=
\lim_{t\to0}\frac{
(\eta+t\rho)(h_0(\omega(\eta+t\rho,b)))-\eta(h_0(\omega(\eta,b)))}{t}.
\end{equation}
The right hand side is easily seen to be
equal to
\begin{equation}\label{right}
\rho(h_0(\omega(\eta,b)))+
(\eta\circ h_0'(\omega(\eta,b)))\left(\lim_{t\to0}\frac{\omega(\eta+t\rho,b)-\omega(\eta,b)}{t}\right)
.
\end{equation}
Fr\'echet differentiating in the variable $b$, we obtain
\begin{equation}\label{frechet}
\frac{\partial\omega(\eta,b)}{\partial b}={\rm Id}_\mathcal B+
\eta\circ h_0'(\omega(\eta,b))\circ
\frac{\partial\omega(\eta,b)}{\partial b},
\end{equation}
where the above is an equality of linear endomorphisms of $\mathcal B$.
$ $From \eqref{gateaux} and \eqref{right} we easily obtain
\begin{equation}\label{eq:8.11}
\lim_{t\to0}\frac{\omega(\eta+t\rho,b)-\omega(\eta,b)}{t}
=\left({\rm Id}_\mathcal B-\eta\circ h_0'(\omega(\eta,b))\right)^{-1}
(\rho(h_0(\omega(\eta,b)))),
\end{equation}
while \eqref{frechet} assures us that the linear operator
${\rm Id}_\mathcal B-\eta\circ h_0'(\omega(\eta,b))$ is
indeed invertible, as
\begin{equation}\label{eq:8.12}
\left({\rm Id}_\mathcal B-\eta\circ h_0'(\omega(\eta,b))\right)
\circ\frac{\partial\omega(\eta,b)}{\partial b}={\rm Id}_\mathcal B.
\end{equation}
Combining \eqref{eq:8.11} and \eqref{eq:8.12} provides us with the
differential equation satisfied by $\omega$, which is of interest in
its own right:
\begin{equation}\label{omega}
\frac{\partial\omega(\eta,b)}{\partial\eta}(\rho)=
\frac{\partial\omega(\eta,b)}{\partial b}(\rho(h_0(\omega(\eta,b)))),
\quad \Im b>0,\eta\in\mathcal{CP}(\mathcal B)\textrm{ invertible, }
\rho\in\mathcal{CP}(\mathcal B).
\end{equation}
We recall the definition of $\omega$ and note that $h(\eta,b)=h_0(\omega(\eta,b))$ in order to obtain
$$
\rho(h(\eta,b))+\eta\left(\frac{\partial h(\eta,b)}{\partial\eta}(
\rho)\right)=\rho(h(\eta,b))+\left(\eta\circ\frac{\partial h(\eta,b)}{\partial b}\right)(\rho(h(\eta,b))),
$$and we conclude by the invertibility of
$\eta$.
\end{proof}

\begin{remark}
Let us justify why maps satisfying the conditions of the above
proposition are important for our paper.

A bit of review: denoting $G_\mu(b)=
b^{-1}+b^{-1}M_\mu(b^{-1})b^{-1}$ (recall Equation \eqref{eq:2.10}),
we observe that
$G_\mu(b)=\mu\left[(b-\mathcal X)^{-1}\right]$ is an extension
of $G_\mu$ to elements $b\in\mathcal B,\Im b>0.$ As noted
in \cite{V2000}, $G_\mu(b)$ is invertible in $\mathcal B$ whenever
$\Im b>0$ and moreover, $\Im(G_\mu(b)^{-1})\ge\Im b$.
We also record here the fact that, as shown in \cite{V1995}, one has
\begin{equation}   \label{functional}
\mu\left[(1+bR_\mu(b)-b\mathcal X)^{-1}\right]=1,\quad
\|b\|\textrm{ small}.
\end{equation}

Now, let us denote $h_\mu(b) := G_\mu(b)^{-1}-b,$ $\Im b>0$. It
follows easily from the definition of $B_\mu$ that
$B_\mu(b)=-h_\mu(b^{-1})$.
Using the definition of the transformation $\mathbb B_\alpha$,
we note that $h_{\mathbb B_\alpha(\mu)}(b)=
(\iB +\alpha)^{-1}h_{\mu^{\boxplus(\iB +\alpha)}}(b)=
h_\mu(b+\alpha(\iB +\alpha)^{-1} h_{\mu^{\boxplus(\iB +\alpha)}}(b))
=h_\mu(b+\alpha h_{\mathbb B_\alpha(\mu)}(b)).$ (We have used
the equation \eqref{eqn:4.92} for the first equality,
equation \eqref{eqn:4.91}, equation \eqref{functional}, as well as
analytic continuation for the second, and direct substitution from
the first for the third equality.) Thus, taking
$h(\eta,b)=h_{\mathbb B_\eta(\mu)}(b)$, provides us with an
example of a map satisfying the conditions of the above proposition.
\end{remark}

\begin{theorem} \label{thm:8.3}
Assume that $X=X^*$ and $S$ are free over $\mathcal B$ and
$S$ is a $\mathcal B$-valued centered semicircular of invertible
variance $\eta.$ If we denote $G(\eta,b)=
\mu_{X+S}\left[(b-\mathcal X)^{-1}\right]$, $\Im b>0$, then
$$
\frac{\partial G(\eta,b)}{\partial\eta}(\rho)+
\frac{\partial G(\eta,b)}{\partial b}(\rho(G(\eta,b)))=0,
$$
where $\rho\in\mathcal{CP}(\mathcal B)$, $\eta\in{\rm Int}(
\mathcal{CP}(\mathcal B))$, and $\Im b>0$. The derivative
with respect to $b$ is in the Fr\'echet sense and the derivative
with respect to $\eta$ is taken in the G\^{a}teaux
sense.
\end{theorem}

\begin{proof}
This is an immediate consequence of Proposition \ref{prop:8.1},
Theorem \ref{thm:7.11} and the above remarks. Indeed, as we know
that the $R$-transform of $S$ is $R_{\mu_S}(b)=\eta(b)$, it follows
that
$$
R_{\mu_{S+X}}(b)=R_{\mu_X}(b)+R_{\mu_S}(b)=R_{\mu_X}(b)+\eta(b),\quad
\|b\|\textrm{ small}.
$$
After adding $b^{-1}$ to both sides of the above equation, the definition of
the $R$-transform in terms of $G_{\mu_X}(b)=
\mu_X\left[(b-\mathcal X)^{-1}\right]$ allows us to re-write it
as $b=\eta\left(G_{\mu_{X+S}}(b)\right)+G_{\mu_X}^{-1}(G_{\mu_{X+S}}
(b)).$ This equation holds when $\Im b>0$ and $\|b^{-1}\|$ is small
enough. Moving $\eta\left(G_{\mu_{X+S}}(b)\right)$ to the left,
composing with $G_{\mu_X}$ on the left and applying analytic
continuation allows us to find the condition of Proposition \ref{prop:8.1} satisfied by $h_0=-G_{\mu_X}$ and $h(\eta,b)=-G_{\mu_{X+S}}
(b)$, for all $b$ with strictly positive imaginary part.
\end{proof}

\subsection{Analytic subordination for $G_{\mu_X^{\boxplus\alpha}}$}

Given $\mu\in\Sigma(\mathbb C)$, an observation important in the
study of the semigroup $\{\mu^{\boxplus t}\colon t\ge1\}$ was
that its Cauchy-Stieltjes transform satisfies an analytic subordination
property in the sense of Littlewood: for each $t\ge1$ there exists
an analytic self-map $\omega_t$ of the complex upper half-plane
so that $G_\mu\circ\omega_t=G_{\mu^{\boxplus t}}$, as shown in
\cite{BB2004,BB2005,C}. We shall present our result in terms of
analytic functions on $\{b\in\mathcal B\colon\Im b>0\}$ as in
\cite{V1995,HRS07}, but it is fairly straightforward to see that all
functions involved have fully matricial extension in the sense of
\cite{V2000}.

\begin{theorem}   \label{thm:8.4}
For any $\mu\in\Sigma^0(\mathcal B)$ and $\alpha\in\mathcal{CP(B)}$ so that $\alpha-1$ is still completely positive there exists an analytic
function $\omega_\alpha\colon\{b\in\mathcal B\colon\Im b>0\}\to
\{b\in\mathcal B\colon\Im b>0\}$ so that $G_\mu\circ\omega_\alpha=
G_{\mu^{\boxplus\alpha}}$. The function $\omega_\alpha$ satisfies
the functional equation
\begin{equation}\label{eq:8.8}
\omega_\alpha(b)=b+(\alpha-1)h_\mu(\omega_\alpha(b)),\quad\Im b>0.
\end{equation}
\end{theorem}

\begin{proof}
Let $f\colon\{b\in\mathcal B\colon\Im b>0\}\times\{b\in\mathcal B\colon\Im b>0\}\to\{b\in\mathcal B\colon\Im b>0\}$ be given by $f(b,w)=
b+(\alpha-1)h_\mu(w)$. As shown in \cite[Remark 2.5]{BPV2010},
$\Im h_\mu(w)\ge0$ whenever $\Im w>0$. Moreover, as it is known from
\cite{PV2010} that $h_\mu(w)=\sigma[(\mathcal X-b)^{-1}]-\mu[\mathcal
X]$ for a completely positive map $\sigma$ of norm equal to the
variance of $\mu$, on the set
$\{w\in\mathcal B\colon\Im w>\Im b/2\}$,
$h_\mu$ is uniformly bounded by $M_b=2\|\mu[\mathcal X]\|+4\|\alpha\|_{
\rm cp}\cdot
\|\mu[\mathcal X\cdot\mathcal X]-\mu[\mathcal X]\cdot\mu[\mathcal X]
\|_{\rm cp}\cdot\|[\Im b]^{-1}\|+2.$ Thus, the map $f(b,\cdot)$ maps
the set $\{w\in\mathcal B\colon\Im w\ge\Im b/2,\|w\|\le2M_b\}$
inside its interior, and so, by \cite[Theorem 3.1]{HRS07}, there
exists a unique fixed point of this map in the interior of this set.
Thus, $\omega_\alpha$ is indeed well-defined. Moreover,
$$
\omega_\alpha(b)=\lim_{n\to\infty}\underbrace{f(b,f(b,\cdots f}_{n
\textrm{ times}}(b,w)\cdots))
$$
for any $w\in M_b$, $\Im b>0$, which means that
$\omega_\alpha$ is locally the uniform limit of
a sequence of maps which are analytic in $b$, and
hence it is analytic itself.

Now, equation \eqref{eq:8.8} is equivalent to the equation
$\alpha\omega_\alpha(b)+(1-\alpha)G_\mu(\omega_\alpha(b))^{-1}=b$,
$\Im b>0$. If
$G_\mu\circ\omega_\alpha=G_{\mu^{\boxplus\alpha}}$ indeed holds,
then we must be able to verify the relation
$\omega_\alpha(b^{-1}+\alpha R_\mu(b))=b^{-1}+R_\mu(b)$;
replacing in the above form of \eqref{eq:8.8} gives
$\alpha(b^{-1}+R_\mu(b))+(1-\alpha)G_\mu(b^{-1}+R_\mu(b))=
b^{-1}+\alpha R_\mu(b)$, a relation trivially true from the definition
of the $R$-transform. All these formulas hold for $b$ invertible
of small enough norm. Analytic continuation allows us to conclude.
\end{proof}

While, like in \cite{BB2005}, we have proved in the above proposition
the existence of the subordination function without any recourse to
any other tool except for analytic function theory, unlike in
\cite{BB2005}, we are not able to conclude from the above the {\em
existence} of $\mu^{\boxplus\alpha}$. The missing ingredient is
a good characterization of maps on the operatorial upper
half-plane which are operator-valued Cauchy-Stieltjes transforms.

\section{Examples}

\begin{example}
For $\lambda \in \cB$ symmetric, we define $\delta_\lambda \in \Sigma^0(\cB)$ by
\[
\delta_\lambda[\cX b_1 \cX \ldots \cX b_n] = \lambda b_1 \lambda \ldots \lambda b_n
\]
or more generally $\mu[P] = P(\lambda)$. Then
\[
\left( \delta_{\lambda} \right)^{\boxplus \alpha} = \left( \delta_{\lambda} \right)^{\uplus \alpha} = \delta_{\alpha[\lambda]}.
\]
\end{example}

\begin{example}
\label{Example:semicircular}
If $\gamma_\eta$ is a centered $\cB$-valued semicircular distribution, then
\[
R_{\gamma_\eta}[\cX b_1 \cX b_2 \ldots b_{n-1} \cX] = \delta_{n, 2} \eta[b_1].
\]
So
\[
\gamma_{\eta} = \left( \gamma_I \right)^{\boxplus \eta}.
\]
\end{example}

In forthcoming work we will describe $\cB$-valued free Meixner distributions, which include the examples above as well as many others.

The following definition generalizes Definition~4.4.1 of \cite{S1998}.

\begin{definition}
\label{Defn:Compound-Poisson}
Let $\nu \in \Sigma(\cB)$, and $\alpha \in \cCP(\cB)$. The generalized free compound Poisson distribution $\pi_{\nu, \alpha}^\boxplus \in \Sigma(\cB)$ is determined by
\[
R_{\pi_{\nu, \alpha}^\boxplus} = (\iB - \alpha) \circ \delta_0 + \alpha \circ \nu,
\]
or more precisely by
\[
R_{\pi_{\nu, \alpha}^\boxplus}^{[n]}(b_1, b_2, \ldots, b_{n-1}) = \alpha \left[ M_\nu^{[n]}(b_1, b_2, \ldots, b_{n-1}) \right].
\]
In particular, $\pi_{\nu, \alpha}^\boxplus = (\pi_{\nu, I}^\boxplus)^{\boxplus \alpha}$. Similarly, generalized Boolean compound Poisson distributions are determined by
\[
B_{\pi_{\nu, \alpha}^\uplus}^{[n]}(b_1, b_2, \ldots, b_{n-1}) = \alpha \left[ M_\nu^{[n]}(b_1, b_2, \ldots, b_{n-1}) \right].
\]
\end{definition}

\begin{example}
Let $\cB = M_2(\bC)$, $\nu = \delta_1$, and
\[
\alpha[b] = \left(\begin{smallmatrix} 0 & 1 \\ 1 & 0 \end{smallmatrix}\right) b \left(\begin{smallmatrix} 0 & 1 \\ 1 & 0 \end{smallmatrix}\right).
\]
Suppose $\pi_{\nu, \alpha}^\boxplus = \pi_{\mu, t}^\boxplus$ for some $\mu \in \Sigma(\cB)$ and $t > 0$. Then
\[
\begin{split}
M_\mu^{[n]}(b_1, b_2, \ldots, b_{n-1})
& = \frac{1}{t} R_{\pi_{\nu, \alpha}^\boxplus}^{[n]}(b_1, b_2, \ldots, b_{n-1}) \\
& = \frac{1}{t} \alpha \left[ M_\nu^{[n]}(b_1, b_2, \ldots, b_{n-1}) \right]
= \frac{1}{t} \alpha[b_1 b_2 \ldots b_{n-1}].
\end{split}
\]
In particular,
\[
\mu[(1 - b \cX)^\ast (1 - b \cX)]
= 1 - b \mu[\cX] - \mu[\cX] b^\ast + \mu[\cX b^\ast b \cX]
= 1 - \frac{1}{t} (b \alpha[1] + \alpha[1] b^\ast) + \frac{1}{t} \alpha[b^\ast b].
\]
So for $\alpha$ as above and $b = \left(\begin{smallmatrix} t & 0 \\ 0 & 0 \end{smallmatrix}\right)$,
\[
\mu[(1 - b \cX)^\ast (1 - b \cX)]
= \left(\begin{smallmatrix}
1 - 2 & 0 \\ 0 & 1 + t
\end{smallmatrix}\right)
\]
is not positive. It follows that such $\mu, t$ do not exist, and so the class of distributions in the preceding definition is wider than Definition~4.4.1 of \cite{S1998}.
\end{example}

We end the paper with an alternative operator model for Boolean compound Poisson distributions.

\begin{example}
Let $\mu$ be a Boolean compound Poisson distribution. Also, let $\alpha : \cB \rightarrow \cB$ be a completely positive map which has the special form $\alpha[b] = (1 + e) b (1 + e^\ast)$ for some $e \in \cB$.

We can choose a noncommutative probability space $(\cM = \cB \oplus \cM_0, \bE, \cB)$ and a
selfadjoint noncommutative random variable $Y\in \cM_0$ satisfying
$$M_Y=B_\mu.$$
By taking a further Boolean product (with amalgamation over $\cB$) we may assume that $\cM_0$ contains an element $Q$ which is Boolean independent from $Y$ over $\cB$ and satisfies
\[
\bE(Q)=e, \qquad \mbox{Var}_Q(b)=b,
\]
where
\[
\mbox{Var}_Q(b)=\bE((1+Q^\ast)b(1+Q))-\bE(1+Q^\ast)b\bE(1+Q).
\]
\end{example}

\begin{theorem}
In the setting of the preceding example, let
\[
T=(1+Q)Y(1+Q^\ast).
\]
Then the distribution of $T$ is $\mu^{\uplus \alpha}$.
\end{theorem}

\begin{proof}
By Boolean independence of $Y$ and $Q$,
\begin{eqnarray*}
M_T^{[1]}&=&\bE(T)\\
         &=&(1+\bE(Q))\bE(Y)(1+\bE(Q)^\ast)\\
         &=&(1+\bE(Q))M_Y^{[1]}(1+\bE(Q)^\ast),
\end{eqnarray*}
so $M_T^{[1]}=\alpha (M_Y^{[1]})$.\\
To compute $M_T^{[n]}(b,\ldots ,b)$ for $n\geq 2$
and $b\in \cB$, we set, $$R_i:=(1+Q^\ast)b_i(1+Q)-b_i = Q^\ast b_i + b_i Q + Q^\ast b_i Q.$$
Note that $R_i \in \cM_0$. Then, since $R_i$ is Boolean independent from $Y$ over $\cB$, one has:
\[
\begin{split}
& M_T^{[n]}(b_1, b_2, \ldots, b_{n-1}) \\
&\qquad =\bE( T b_1 T b_2 \cdots T b_{n-1} T)\\
&\qquad =\bE((1+Q)Y(1+Q^\ast) b_1 (1+Q)\cdots Y(1+Q^\ast) b_{n-1} (1+Q)Y(1+Q^\ast))\\
&\qquad =\bE((1+Q)Y(b_1 + R_1)Y\cdots (b_{n-1} + R_{n-1})Y(1+Q^\ast))\\
&\qquad =\sum _{1\leq m_1 < \ldots < m_p=n} \bE \Bigl((1+Q)(Y b_1 \ldots b_{m_1 - 1} Y) R_{m_1} \\
&\qquad\qquad                              \cdots R_{m_1 + \ldots + m_{p-1}} (Y b_{m_1 + \ldots + m_{p-1} + 1} \ldots b_{n-1} Y)(1+Q^\ast) \Bigr)\\
&\qquad =\sum _{1\leq m_1 < \ldots < m_p=n} \bE (1+Q) \bE(Y b_1 \ldots b_{m_1 - 1} Y) \bE(R_{m_1}) \\
&\qquad\qquad                              \cdots \bE(R_{m_1 + \ldots + m_{p-1}}) \bE(Y b_{m_1 + \ldots + m_{p-1} + 1} \ldots b_{n-1} Y) \bE(1+Q^\ast) \\
&\qquad =\sum _{1\leq m_1 < \ldots < m_p=n} \bE (1+Q) M_Y^{[m_1]}(b_1, \ldots, b_{m_1 - 1}) \bE(R_{m_1}) \\
&\qquad \qquad                              \cdots \bE(R_{m_1 + \ldots + m_{p-1}}) M_Y^{[m_p]}(b_{m_1 + \ldots + m_{p-1} + 1}, \ldots, b_{n-1}) \bE(1+Q^\ast).
\end{split}
\]
Now using
\[
\bE(R_i) = \bE((1+Q^\ast)b_i(1+Q)) - b_i = \mbox{Var}_Q(b_i) + \bE(1+Q^\ast)b_i\bE(1+Q) - b_i = \bE(1+Q^\ast)b_i\bE(1+Q),
\]
we get
\[
\begin{split}
& M_T^{[n]}(b_1, b_2, \ldots, b_{n-1}) \\
&\qquad =\sum _{1\leq m_1 < \ldots < m_p=n} \Bigl( \bE (1+Q) M_Y^{[m_1]}(b_1, \ldots, b_{m_1 - 1}) \bE(1 + Q^\ast) \Bigr) b_{m_1} \\
&\qquad \qquad                              \cdots b_{m_1 + \ldots + m_{p-1}} \Bigl( \bE(1 + Q) M_Y^{[m_p]}(b_{m_1 + \ldots + m_{p-1} + 1}, \ldots, b_{n-1}) \bE(1+Q^\ast) \Bigr).
\end{split}
\]
Using M{\"o}bius inversion and formula \eqref{eqn:4.51}, it follows that
\[
\begin{split}
B_T^{[n]}(b_1, b_2, \ldots, b_{n-1})
& = \bE (1+Q) M_Y^{[n]}(b_1, \ldots, b_{n-1}) \bE (1+Q^\ast) \\
& = (1+e) M_Y^{[n]}(b_1, \ldots, b_{n-1}) (1+e^\ast) \\
& = (1+e) B_\mu^{[n]}(b_1, \ldots, b_{n-1}) (1+e^\ast) \\
& = \alpha(B_\mu^{[n]}(b_1, \ldots, b_{n-1})).
\end{split}
\]
This proves that the distribution of $T$ is $\mu^{\uplus \alpha }$.
\end{proof}

$\ $

$\ $

\textbf{Acknowledgment.} This work was started when the authors
participated in a Research in Teams project (10rit159) at the
Banff International Research Station, in August 2010. The support
of BIRS and its very inspiring environment are gratefully
acknowledged.

$\ $

%\bibliographystyle{amsplain}
%\bibliography{bibtemp}

\providecommand{\bysame}{\leavevmode\hbox to3em{\hrulefill}\thinspace}
\providecommand{\MR}{\relax\ifhmode\unskip\space\fi MR }
% \MRhref is called by the amsart/book/proc definition of \MR.
\providecommand{\MRhref}[2]{%
  \href{http://www.ams.org/mathscinet-getitem?mr=#1}{#2}
}
\providecommand{\href}[2]{#2}

\end{document}